\documentclass[12pt,reqno]{amsart}
\usepackage{amsmath,amsxtra,amssymb,amsthm,amsfonts}
\makeatletter
\newcommand*{\rom}[1]{\expandafter\@slowromancap\romannumeral #1@}
\makeatother

\usepackage[margin=1in]{geometry}
\usepackage{color}

\usepackage[pdftex,linktocpage=true,colorlinks,citecolor=blue,linkcolor=blue,pagebackref]{hyperref}
\usepackage{hyperref}
\usepackage{cleveref}
\usepackage[alphabetic,backrefs]{amsrefs}

\numberwithin{equation}{section}

\renewcommand{\Re}{\text{Re}}

\newcommand{\ten}{\otimes}

\newcommand{\pl}{\hspace{.1cm}}

\newcommand{\ran}{\rangle}
\newcommand{\lan}{\langle}
\newcommand{\hten}{\widehat{\ten}}
\newcommand{\al}{\alpha}
\newcommand{\si}{\sigma}

\newcommand{\la}{\lambda}

\newcommand{\id}{\iota_{\infty,2}^n}

\renewcommand{\id}{{\operatorname{id}}}

\renewcommand{\L}{{\mathcal L}}
\newcommand{\F}{{\mathcal F}}

\newcommand{\M}{{\mathcal M}}

\newcommand{\K}{{\mathcal K}}

\newcommand{\N}{{\mathcal N}}

\newcommand{\norm}[2]{\parallel \! #1 \! \parallel_{#2}}

\newtheorem{lemma}{Lemma}[section]
\newtheorem{prop}[lemma]{Proposition}
\newtheorem{theorem}[lemma]{Theorem}
\newtheorem{cor}[lemma]{Corollary}

\newtheorem{rem}[lemma]{Remark}

\newcommand{\re}{\begin{rem}\rm}
\newcommand{\mar}{\end{rem}}

\newcommand{\bra}[1]{\langle{#1}|}
\newcommand{\ket}[1]{|{#1}\rangle}

\newcommand{\prf}{\begin{proof}[\bf Proof:]}

\newcommand{\xspace}{\hbox{\kern-2.5pt}}

\newcommand{\tr}{\tau}

\allowdisplaybreaks
\newtheorem{defi}[lemma]{Definition}
\begin{document}
\title{Quantum Majorization on Semifinite von Neumann algebras}
\author[P.~Ganesan]{Priyanga Ganesan\textsuperscript{1}}
\thanks{\textsuperscript{1}Department of Mathematics, Texas A\&M University, USA}
\email{priyanga.g@tamu.edu}
\author[L.~Gao]{Li Gao\textsuperscript{1}}
\email{ligao@math.tamu.edu}
\author[S.~K. ~Pandey]{Satish K. Pandey\textsuperscript{2}}
\thanks{\textsuperscript{2}Faculty of Mathematics, Technion-Israel Institute of Technology, Haifa 3200003, Israel}
\email{satishpandey@campus.technion.ac.il}
\author[S.~Plosker]{Sarah Plosker\textsuperscript{3}}
\thanks{\textsuperscript{3}Department of Mathematics and Computer Science, Brandon University,
Brandon, MB R7A 6A9, Canada}
\email{ploskers@brandonu.ca}

\subjclass[2010]{Primary 46L07; Secondary 47L07, 
47L90, 91A17.}

\begin{abstract}We extend Gour \emph{et al}'s characterization of quantum majorization via conditional min-entropy to the context of semifinite von Neumann algebras. Our method relies on a connection between conditional min-entropy and operator space projective tensor norm for injective von Neumann algebras. This approach also connects the tracial Hahn-Banach theorem of Helton, Klep
and McCullough to noncommutative vector-valued $L_1$-space.
\end{abstract}
\maketitle
\section{Introduction}

Majorization is a fundamental tool introduced by Hardy, Littlewood, and Polya \cite{hardy} that finds application in various fields \cite{marshall}. Among the different motivations for majorization, the core idea is a notion of ``disorder''. For example, a probability distribution is majorized by another if it is less deviated from the
uniform distribution. Recently, Gour, Jennings, Buscemi, Duan, and Marvian in \cite{gour18} use the concept of ``quantum majorization'' to accommodate the ordering of states and processes in quantum mechanical systems.

Let $H$ be a finite dimensional Hilbert space and $B(H)$ be space of the bounded operators acting on $H$. A density operator $\rho\in B(H)$  (called a  state on the quantum system $H$ in the quantum information theory literature) is positive and has  trace $1$ . The process between quantum systems is modeled by completely positive trace preserving maps (also called quantum channels) which map density operators to density operators. For two bipartite density operators $\rho$ and $\si$ on the tensor product Hilbert space $H_A\ten H_B$, $\si$ is said to be quantum majorized by $\rho$ if there exists a linear completely positive trace preserving (CPTP) map  $\Phi:B(H_B)\to B(H_B)$ such that $\si=id\ten \Phi(\rho)$. This concept has been studied in different contexts under various guises \cite{Shmaya, Chefles, buscemi, bds, jencovachannel}. Intuitively, quantum majorization describes the disorder observed from the $B$ system. This can be witnessed from the data processing inequality of conditional entropy,
\[H(A|B)_{\rho}\le H(A|B)_{id\ten \Phi(\rho)}=H(A|B)_{\si}\pl,\]where $H(A|B):=H(\rho)-H(\tr\ten \id(\rho))$ and $H(\rho)=-\tr(\rho\log \rho)$ is the von Neumann entropy. The conditional entropy $H(A|B)_\rho$ describes the uncertainty of the bipartite density operator $\rho$ given its information on the $B$ system \cite{winter}. The data processing inequality says such uncertainty is monotone non-decreasing under quantum majorization. As a converse to data processing inequality, Gour and his coauthors proved the following characterization of quantum majorization using conditional min-entropy $H_{min}(A|B)$, defined as
\begin{align}H_{min}(A|B)_{\rho}=&-\log \inf\{ \tr(\omega) | \rho \le \la 1\ten \omega\pl \text{for some positive} \pl \omega\in B(H_B)\}. \label{def}
\end{align}
\noindent{\bf Theorem (\cite{gour18}).} \emph{Let $H_A,H_B$ be finite dimensional Hilbert spaces. For two bipartite density operators  $\rho$ and $\si$, $\si$ is quantum majored by $\rho$ if and only if for all finite dimensional $H_A'$ and all CPTP maps $\Psi:B(H_A)\to B(H_{A'})$ ,
\begin{align}H_{min}(A'|B)_{\Psi\ten \operatorname{id}(\rho)}\le H_{min}(A'|B)_{\Psi\ten \operatorname{id}(\si)}\pl.\label{eqv}\end{align}}
$H_{min}(A|B)$ is analogy of $H(A|B)$ as the R\'enyi $p$-version at $p=\infty$ \cite{muller} and it connects to $H(A|B)$ by the quantum version of asymptotic equipartition
property \cite{tomamichel}. The ``only if'' direction in the above theorem follows from the data processing inequality of $H_{min}$, which is indeed self-evident from its definition \eqref{def}. The other direction states
that quantum majorization is actually determined by the data processing inequality of $H_{min}$.
In \cite{gour18}, the above theorem has been used to characterize quantum process under group symmetry and thermodynamic condition. It has further extensions from bipartite states to bipartite quantum channels \cite{gour19}.

In this work, we revisit Gour \emph{et al}'s theorem from a functional analytic perspective. Our starting point is the observation that the conditional min-entropy corresponds to the operator space tensor norm \begin{align}H_{min}(A|B)_\rho=-\log \norm{\rho}{S_1(H_B)\hten B(H_A)}\label{ana}\end{align} where $S_1(H_B)$ is the set of  trace class operators on $H_B$ and $S_1(H_B)\hten B(H_A)$ is the operator space projective tensor product.  This correspondence is based on an factorization expression for the norm of $S_1(H_B)\hten B(H_A)$ that Pisier used in \cite{pisier93} to define noncommutative vector-valued $L_p$ space. On the other hand, it is known \cite{ER00,PB} that the dual space of $S_1(H_B)\hten B(H_A)$ is the completely bounded maps $CB(B(H_A),B(H_B))$, where quantum channels correspond to unital completely positive maps by taking adjoints. From this perspective, $H_{min}$ is the dual of CB norm with respect to quantum channels and Gour \emph{et al}'s theorem is essentially a Hahn-Banach separation theorem. Using this approach, we prove the following characterization of quantum majorization using projective tensor norm which extend Gour \emph{et al}'s results to the setting of tracial von Neumann algebra. We consider two semifinite von Neumann algebras $\M$ and $\N$ equipped with normal faithful semi-finite traces $\tau_\M$ (resp. $\tau_\N$). We denote  $L_1(\M)$ (resp. $L_1(\N)$) as the space of $1$-integrable operators with respect to $\tau_\M$ (resp. $\tau_\N$). Our main theorem is\newpage
\begin{theorem}[c.f. Theorem \ref{von}]
Let $\M$ and $\N$ be two semifinite von Neumann algebras. Suppose $\M$ is injective. Then for two density operators $\rho,\si \in L_1(\M\overline{\ten} \N)$, there exists a CPTP map $\Phi:L_1(\M)\to L_1(\M)$ such that $\Phi\ten \operatorname{id}(\rho)= \si$ if and only for any projection $e\in \M$ with $\tau_\M(e)<\infty$ and for any CPTP map $\Psi:L_1(\N)\to L_1(e\M e^{op})\cap e\M e^{op}$,
    \[ \norm{\operatorname{id}\ten \Psi(\rho)}{L_1(\M)\hten e\M e^{op}}\ge \norm{\operatorname{id}\ten \Psi(\si)}{L_1(\M)\hten e\M e^{op}}\]
\end{theorem}
Here the $L_1(\M)\hten \N$-norm gives the analogue of $H_{min}$ as in \eqref{ana}. We note that the assumption on injectivity is crucial in our argument. Indeed, we show that for semifinite von Neumann algebras, the conditional min-entropy $H_{min}$ coincides with the projective tensor norm $L_1(\M)\hten \N$ if and only if $\M$ is injective. This can be viewed as a predual form of Haagerup's characterization of injectivity via decomposablity \cite{haagerup84}. Beyond injectivity, it is not clear whether the above equivalence holds and we do not know the information-theoretic meaning of the projective tensor norm.

The above theorem admits several variants. By taking $\N= l_\infty$, the commutative von Neumann algebra of bounded sequences, Theorem 1.1 concerns the quantum interpolation
problem of converting an infinite family of density operators into another family of density operators using a CPTP map. On the other hand, the dual form of Theorem 1.1 provides a characterization for the  factorization of CPTP maps (a problem known as channel factorization).  A CPTP map $S$ is quantum majorized by $T$ if  $S$ admits a factorization $S=\Phi\circ T$ for some CPTP map $\Phi$. Note that in finite dimensions, quantum majorization applies to CPTP maps via their Choi matrices. However, in infinite dimensions, the Choi matrix of a CPTP map is never trace class and our dual consideration is needed. Inspired by Jenvoca's work \cite{jencovachannel} on statistical deficiency for CPTP maps, we also consider the approximate case when the error $\id\ten\Phi(\rho)-\si$ is small but non-zero.

Our approach also has applications to the tracial Hahn-Banach theorem in \cite{tracial}. The tracial Hahn-Banach theorem is a dual form of Effros-Wrinkler's separation theorem for matrix convex sets. We find that the duality behind the tracial Hahn-Banach theorem is the same duality  as that between the operator space projective tensor product and completely bounded maps. Using an idea similar to the one in characterization of quantum majorization, we give  a tracial Hahn-Banach theorem on $L_1(\M)\hten E$ for a semi-finite injective von Neumann algebra  $\M$ and an arbitrary operator space $E$. If we  replace $L_1(\M)$ by an abstract operator space, our method gives some analogous results under the assumptions of $1$-locally reflexivity and completely contractive approximation property.

Our paper is organized as follows. Section \ref{sec:prelim} reviews some basic operator space theory needed for the remainder of the paper. In Section \ref{VNA}, we first discuss the relation between $H_{min}$ and projective tensor norm and the connection to injectivity of von Neumann algebras. After that, we prove our main theorem and its variants with respect to  channel factorization and the approximate case. In particular, all the results in this section apply to $B(H)$ with $H$ being infinite dimensional. As this is arguably the case of most interest in  quantum information theory, we summarize the implications for $B(H)$ in Section \ref{BH case}. Section \ref{tracialset} is devoted to the tracial Hahn-Banach theorem and the connection to noncommutative vector-valued $L_1$ space. Section \ref{sec:opspace} discusses the parallel results on projective tensor product of abstract operator spaces.

\section{Operator Space Preliminaries}\label{sec:prelim}
In this section we briefly recall some operator space basics that are needed in our discussion. We refer to the books \cite{pisieros,ER00} for more information on operator space theory. We denote by $B(H)$ the bounded operator on a complex Hilbert space $H$ and $M_n:=M_n(\mathbb C)$  the algebra of  $n\times n$ complex-valued matrices. A (concrete) operator space $E$ is a closed subspace of some $B(H)$. We denote by $M_n(E)$ the set of $n \times n$ matrices with entries from $E$ and similar $M_{n,m}(E)$ for $n\times m$ rectangular matrix. The space $M_n(B(H))$ is naturally isomorphic to $B(H^{(n)})$, where $H^{(n)}=\ell^n_2(H)$ is the Hilbert space direct sum of $n$ copies of $H$. For all $n\ge 1$, the inclusion $M_{n}(E)\subset M_n(B(H))\cong B(H^{(n)})$ induces a norm
on the matrix level space $M_n(E)$ which we denote by $\norm{\cdot}{M_n(E)}$. The operator space structure of $E$ is given by the norm sequence   $\norm{\cdot}{M_n(E)}, n\ge 1$.

Given a linear map $u:E\to F$ between two operator spaces $E$ and $F$,  $u$ is \emph{completely bounded} (or $CB$) if its completely bounded norm ($CB$-norm)
\[\norm{u}{cb}:=\sup_{n\ge 1} \norm{\operatorname{id}_n\ten u:M_n(E)\to M_n(F)}{op}\]
is finite. Here $id_n$ is the identity map on $M_n$. We say $u$ is a \emph{complete isometry} if for each $n$, $id_n\ten u$ is an isometry. We denote by $CB(E,F)$ the Banach space of all completely bounded maps $E\rightarrow F$ equipped with the $CB$-norm. Moreover, $CB(E,F)$ is again an operator space with the operator space structure given by $M_n(CB(E,F))=CB(E,M_n(F))$. In particular, the operator space dual is defined as
\[E^*=CB(E,\mathbb{C})\pl.\]
Throughout the paper, we will use $\ten$ for algebraic tensor product. Given two operator spaces $E\subset B(H_A)$ and $F\subset B(H_B)$, the operator space injective tensor product $E\ten_{min} F$ is defined by the (completely) isometric embedding
\begin{align}E\ten_{min} F\subset B(H_A\ten_2 H_B)\pl \label{tensor}\end{align}
where $H_A\ten_2 H_B$ is the Hilbert space tensor product. Namely, $E\ten_{min} F$ is the norm completion of $E\ten F$ for the inclusion $E\ten F\subset B(H_A\ten_2 H_B)$. Via injectivity, one has the (completely) isometric embedding
\begin{align}E^*\ten_{min} F\subset CB(E,F)\pl.\label{inje}\end{align}

Another important tensor product for our work is the projective tensor product. We denote by $\norm{\cdot}{HS}$   the Hilbert-Schmidt norm. The operator space projective tensor product $E\widehat{\ten}F$ is defined as the completion of $E\ten F$ with respect to the following norm,
\begin{align*}
\norm{z}{E\widehat{\ten}F}=\inf \norm{a}{HS} \norm{x}{M_l(E)}\norm{y}{M_m(F)}\norm{b}{HS}
\end{align*}
where the infimum runs over all factorizations of rectangular matrices $a,b$ and $x=(x)_{i,j=1}^m\in M_n(E), y=(y_{pq})_{p,q=1}^m\in M_m(F)$ such that
\begin{align}z=\sum_{i,j=1}^l\sum_{p,q=1}^m a_{i,p}x_{i,j}\ten y_{pq} b_{j,q}\label{b}\pl.
\end{align} For $z=(z_{rs})_{r,s=1}^n\in M_n(E\ten F)$, we consider the following factorization
\begin{align}z_{rs}=\sum_{i,j=1}^l\sum_{p,q=1}^m a_{r,ip}x_{i,j}\ten y_{pq} b_{jq,s}\pl, \label{a}\end{align}
where $a\in M_{n,ml}, b\in M_{ml,n}$ and $x\in M_l(E),y\in M_m(F)$. The operator space structure of $E\hten F$ is defined as
\begin{align*}
\norm{z}{M_n(E\widehat{\ten}F)}=\inf \norm{a}{M_{n,ml}} \norm{x}{M_l(E)}\norm{y}{M_m(F)}\norm{b}{M_{ml,n}}
\end{align*}
where the infimum runs over all factorizations in \eqref{a}. An equivalent characterization is the following duality \cite{ER00,PB}
\begin{align}(E\widehat{\ten} F)^*\cong CB(E,F^*) \pl. \label{pro}\end{align}
For $x\in E, y\in F$ and $\Phi\in CB(E,F^*)$. The dual pairing is
\[\lan x\ten y, \Phi\ran=\lan \Phi(x),y\ran_{(F^*,F)} \pl.\]

Let us mention some basic examples related to our discussion. Let $\K(H)$ denote the space of compact operators on $H$ and $S_1(H)$ the space of trace class operators. We have the operator space dual relations
\begin{align}S_1(H)^*=B(H)\pl, \pl \K(H)^*=S_1(H) \pl, \label{dual}\end{align}
where both dual pairings are given by the trace
\begin{align*}\lan b,a\ran_{(B(H),S_1(H))}=\tr(b^ta)\pl, \lan a,c\ran_{(S_1(H),K(H))}=\tr(a^tc)\label{trans}\end{align*}
where $a^t$ is the transpose of $a$ with respect to a (fixed) orthonormal basis. For two Hilbert spaces $H_A$ and $H_B$,
by \eqref{tensor} and \eqref{inje}
we have the isometric embedding  \[  B(H_A)\ten_{min}B(H_B)\subset B(H_A\ten_2 H_B)\pl, B(H_A)\ten_{min}B(H_B)\subset CB(S_1(H_A),B(H_B)).\] Indeed, one has the equality
\begin{align}\label{iso}B(H_A\ten_2 H_B)\cong CB(S_1(H_A),B(H_B)).\end{align}
Note that by \eqref{pro} and \eqref{dual},
\[CB(S_1(H_A),B(H_B))=S_1(H_A)\widehat{\ten}S_1(H_B)^*\pl, \quad B(H_A\ten_2 H_B)=S_1(H_A\ten_2 H_B)^*\pl.\]
For preduals, $S_1(H_A)\widehat{\ten}S_1(H_B)\cong S_1(H_A\ten_2 H_B)$.

Another example related to our discussion is the space $S_1(H_B)\widehat{\ten} B(H_A)$. Let $S_2(H)$ denote the Hilbert-Schmidt operators on $H$. The operator space projective tensor norm on $S_1(H_B)\widehat{\ten} B(H_A)$ admits the following expression (c.f. \cite{pisier93}) for $x\in S_1(H_B)\ten B(H_A)$
\[\norm{x}{S_1(H_B)\widehat{\ten} B(H_A)}=\inf_{x=(a\ten 1) y(1\ten b)}\norm{a}{S_2(H_B)}\norm{b}{S_2(H_B)}\norm{y}{B(H_B)\ten_{min}B(H_A)}\]
where the infimum is taken over all possible factorizations of $x=(a\ten 1_A)y(b\ten 1_A)$ with $a,b\in S_2(H_B)$ and $1_A$ denotes the identity operator on $H_A$.   For positive $x$, it suffices to choose $a=b^*$ and, by rescaling $\norm{a}{2}=1$, we obtain
\begin{align*} \norm{x}{S_1(H_B)\widehat{\ten} B(H_A)}&= \inf\{ \norm{y}{B(H_B)\ten_{min}B(H_A)}\,|\, x=(a\ten 1)y(a^*\ten 1)\pl \text{for some $\norm{a}{S_2(H_B)}=1$} \}
\\& =\inf \{\la \,|\, x\le \la 1\ten \si \pl \text{for some density operator $\si\in S_1(H_B)$}\}\pl.
\end{align*}
Therefore, this norm on $S_1(H_B)\widehat{\ten} B(H_A)$ corresponds to the conditional min entropy $H_{min}$. That is, for a bipartite density operator $\rho$,
\[H_{min}(A|B)_\rho=-\log \norm{\rho}{S_1(H_B)\widehat{\ten} B(H_A)}\pl.\]
At the dual level, by \eqref{pro} we have
\begin{align}\label{produal} (S_1(H_B)\widehat{\ten} B(H_A))^*=CB(B(H_A),B(H_B))\pl.\end{align}
Note that a CPTP map $\Phi:S_1(H_B)\to S_1(H_A)$ is completely positive trace preserving, and hence
\[ \Phi\in CB(S_1(H_B),S_1(H_A))\subset CB(B(H_A),B(H_B)) \]
where $CB(S_1(H_B),S_1(H_A))\subset CB(B(H_A),B(H_B))$ as normal cb maps by taking adjoints. Therefore, the $S_1(H_B)\widehat{\ten} B(H_A)$ norm or equivalently $H_{min}$, is the dual of CB-norm with respect to quantum channels. This duality is implicitly used in Gour et al's arguments in \cite{gour18}. In quantum information literature, the $CB$-norm of $CB(S_1(H_B),S_1(H_A))$ is also called diamond norm. The diamond norm and its dual norm has been used by Jen\v{c}ov\'a in studying Le Cam's deficiency for quantum channels \cite{jencovachannel}.

\section{Quantum majorization on von Neumann algebras}\label{VNA}
\subsection{$H_{min}$ and injectivity of von Neumann algebras} \label{Hmin relation}
 We first discuss the connection between the conditional min entropy $H_{min}$ and the projective tensor product in the setting of tracial von Neumann algebras. Throughout this paper, we assume that $(\M,\tau_\M)$ and $(\N,\tau_\N)$ are semifinite von Neumann algebras with normal faithful semifinite traces $\tau_\M$ (resp. $\tau_\N$). We introduce the notation
 \[\M_0:=\cup_{e} e\M e\pl,\]
 where the union runs over all projections with $\tau_\M(e)<\infty$ which forms a lattice.
 For $1\le p< \infty$, the space $L_p(\M)$ is the completion of $\M_0$ with respect to the $L_p$-norm
 \[\norm{a}{L_p(\M)}=\tau_\M(|a|^{p})^{1/p}\pl, a\in \M_0\pl.\]
 We will often use the shorthand notation $\norm{\cdot}{p}$ for the $p$-norm and $\norm{\cdot}{\infty}$ for the operator norm in $\M$. Let $\M^{op}=\{a^{op}| a\in \M\}$ be the opposite algebra equipped with reversed multiplication $a^{op}\cdot b^{op}=(ba)^{op}$ and trace $\tau_{\M^op}(a^{op})=\tau_\M(a)$.
The predual of $\M$ can be identified with $\M_*=L_1(\M^{op})$, via the pairing $\lan a^{op},b \ran=\tau_\M(ab)$ for  $a\in L_1(\M)$ and $b\in \M$.
Recall that the von Neumann algebra tensor product $\M\overline{\ten} \N$ is the weak$^*$-closure of $\M\ten_{min}\N$.
The Effros-Ruan isomorphism \cite{ER00} gives a complete isometry
\begin{align}\N\overline{\ten}\M \cong CB(\N_*, \M)\cong CB(L_1(\N^{op}), \M)\pl.\label{er}\end{align}
This isomorphism is order preserving. Indeed, a positive operator $x\in \N\overline{\ten}\M$ corresponds to a completely positive map $T_x\in CB(L_1(\N^{op}), \M)$. As for the predual of \eqref{er},
we have \[L_1(\M)\hten L_1(\N)=L_1(\M\overline{\ten }\N)=(\M^{op}\overline{\ten }\N^{op})_*\pl.\]

The conditional min entropy $H_{min}$ is related to  the vector-valued $L_1$-space introduced in \cite{pisier93}.
We will use the shorthand notation that for $a,b\in \M, y\in \M\overline{\ten }\N$,
\[a\cdot y \cdot b:=(a\ten 1_\N) y (b\ten 1_\N).\]
We define the $L_1(\M,L_\infty(\N))$ norm for $x\in \M_0\ten \N$ as follows,
\begin{align*}
\norm{x}{L_1(\M,L_\infty(\N))}=\inf \{\norm{a}{L_2(\M)}\norm{y}{\M\overline{\ten} \N}\norm{b}{L_2(\M)}|\, x=a\cdot y\cdot b\pl, a,b \in \M_0, y\in \M\ten \N\},
\end{align*}
where the infimum is over all factorizations $x=a\cdot y\cdot b$.
Then $L_1(\M,L_\infty(\N))$ is defined as the completion of $\M_0\ten \N$ under the above norm. The triangle inequality for this norm is verified in \cite[Lemma 3.5]{pisier93}.
We will also use the shorthand notation \[\M\overline{\ten} \N_0=\cup_{q}\M \overline{\ten} q\N q\subset \M\overline{\ten} \N\pl,\] where the union runs over all projections $q\in \N$ with $\tau_\N(q)<\infty$.
For $x\in \M\overline{\ten}\N_0$, we define the $L_\infty(\M,L_1(\N))$ norm as
\begin{align*}&\norm{x}{L_\infty(\M,L_1(\N))}=\sup \{\norm{a\cdot x \cdot b}{L_1(\M\overline{\ten} \N)}| \norm{a}{L_2(\M)}=\norm{b}{L_2(\M)}=1\pl \}\pl.
\end{align*}
This norm clearly satisfies the triangle inequality.  The space $L_\infty(\M,L_1(\N))$ is defined as the norm completion of $\M\overline{\ten}\N_0$. Both spaces contain the corresponding algebraic tensor
\[L_1(\M)\ten \N\subset L_1(\M, L_\infty(\N))\pl, \pl \M\ten L_1(\N)\subset L_\infty(\M, L_1(\N)).\]
Indeed, for $a\ten b$ with $a\in L_1(\M)$ and $b\in \N$, let $e_n$ be the spectral projection of $|a|$ for the interval $[1/n,n]$. Then $e_nae_n\ten b$ converges in $L_1(\M,L_\infty(\N))$ and the limit can be identified with $a\ten b$. It is clear from the definitions that
\begin{enumerate}
\item[i)] a complete contraction $T:L_\infty(\N_1)\to L_\infty(\N_2)$ extends to a contraction \[\operatorname{id}_\M\ten T:L_1(\M,L_\infty(\N_1))\to  L_1(\M,L_\infty(\N_2))\pl.\]
\item[ii)] a complete contraction $S:L_1(\N_1)\to L_1(\N_2)$ extends to a contraction \[\operatorname{id}_\M\ten S:L_\infty(\M,L_1(\N_1))\to  L_\infty(\M,L_1(\N_2))\pl.\]
\end{enumerate}
For the trivial case $\N=\mathbb{C}$, we have $L_1(\M,\mathbb{C})=L_1(\M)$  and $L_\infty(\M,\mathbb{C})=L_\infty(\M)$. In general,
$L_\infty(\M,L_1(\N))$ is a subspace of $\Big( L_1(\M,L_\infty(\N))\Big)^*$. Indeed,
\begin{align*}
\norm{x}{L_\infty(\M,L_1(\N))}=&\sup \{\norm{a\cdot x \cdot b}{1} \pl |\pl  \norm{a}{2}=\norm{b}{2}=1, a,b\in \M_0  \}
\\=&\sup \{|\tau\big(y (a\cdot x \cdot b)\big)|  \norm{a}{2}=\norm{b}{2}=1, a,b\in \M_0, \norm{y}{\M\overline{\ten}\N}=1  \}
\\ = &\sup \{|\tau\big((b\cdot y \cdot a)x\big)|  \norm{a}{2}=\norm{b}{2}=1, a,b\in \M_0, \norm{y}{\M\overline{\ten}\N}=1  \}\\
= &\sup \{|\tau(zx)|  \norm{z}{L_1(\M,L_\infty(\N))}=1  ,z\in \M_0\ten \N\}\pl.
\end{align*}
Here and in the following we will use $\tau:=\tau_\M\ten \tau_\N$ for the product trace.
\begin{lemma}\label{positive} i) For any self-adjoint $x\in \M_0\ten \N$,
\begin{align*} \norm{x}{L_1(\M,L_\infty(\N))}=\inf \{\norm{a}{L_2(\M)}\norm{y}{\M\overline{\ten} \N}\norm{a^*}{L_2(\M)}| \pl {x=a\cdot y\cdot a^* }\pl, a\in \M_0, y\pl \text{self-adjoint}\}.
\end{align*}
ii) For any positive $x \in \M \overline{\ten}\N_0$,
\begin{align*}
\norm{x}{L_\infty(\M,L_1(\N))}=\sup \{ \tau(a\cdot x\cdot a^*) |\pl  \norm{a}{L_2(\M)}=1 \pl \}.
\end{align*}
\end{lemma}
\begin{proof}
For ii),   H\"older's inequality gives,
\begin{align*}\norm{x}{L_\infty(\M, L_1(\N))}= &\sup_{\norm{\pl a\pl }{2}=\norm{\pl b\pl }{2}=1} \norm{(a\ten 1)x(b\ten 1)}{1}\\ \le & \sup_{\norm{\pl a\pl }{2}=1} \norm{(a\ten 1)x^{\frac{1}{2}}}{2}\sup_{\norm{\pl b\pl }{2}=1} \norm{x^{\frac{1}{2}}(b\ten 1)}{2}
\\ =& \sup_{\norm{\pl a\pl }{2}=1} \norm{(a\ten 1)x (a^*\ten 1)}{1}^{\frac{1}{2}}\sup_{\norm{\pl b\pl }{2}=1} \norm{(b^*\ten 1)x (b\ten 1)}{1}^{\frac{1}{2}}\\ =& \sup_{\norm{\pl a\pl }{2}=1} \norm{(a\ten 1)x (a^*\ten 1)}{1}=\sup_{a}\tau(a\cdot x\cdot a^*).
\end{align*}
For i), choose $x=(a\ten 1)y(b\ten 1)$ such that $a,b\in e\M e$ and
\[ \norm{a}{L_2(\M)}=\norm{b}{L_2(\M)}=1, \norm{y}{\M \overline{\ten}\N}<\norm{x}{L_1(\M,L_\infty(\N))}+\epsilon. \]
Take $d=(aa^*+b^*b+\delta e)^{\frac{1}{2}}$. Then $d> 0$ is invertible in $e\M e$ and $\norm{d}{2}=(2+\delta \tau(e))^\frac{1}{2}$. Note that $x=x^*$ implies that \begin{align*}
x&=\frac{1}{2}\Big(a\cdot y \cdot b+b^*\cdot y^*\cdot a\Big)
\\ &= \frac{1}{2} d\cdot \Big( d^{-1}a \cdot y \cdot bd^{-1}+  d^{-1}b^* \cdot y^* \cdot ad^{-1}\Big)\cdot d
\\ &=  d\cdot \tilde{y} \cdot d,
\end{align*}
where
\begin{align*}
\tilde{y}&=\frac{1}{2}\Big( d^{-1}a \cdot y \cdot bd^{-1}+  d^{-1}b^* \cdot y^* \cdot ad^{-1}\Big)
\\ &= \frac{1}{2}\left[\begin{array}{cc} d^{-1}a & d^{-1}b^*
\end{array}\right]
\cdot
\left[\begin{array}{cc} 0 &y\\ y^*&0
\end{array}\right]\cdot
\left[\begin{array}{c} a^*d^{-1}\\ bd^{-1}
\end{array}\right].
\end{align*}
Since $\left\|\left[\begin{array}{cc} 0 &y\\ y^*&0
\end{array}\right]\right\|_{M_2(\M)}=\norm{y}{\M}$, $\norm{\left[\begin{array}{cc} d^{-1}a & d^{-1}b^*
\end{array}\right]}{M_{1,2}(\M)}=\norm{d^{-1}(aa^*+b^*b)d^{-1}}{\M}\le 1$ and similarly $\left\|\left[\begin{array}{c} a^*d^{-1}\\ bd^{-1}
\end{array}\right]\right\|_{M_{2,1}(\M)}\le 1$, it follows that
\[ \norm{\tilde{y}}{\infty}\le \frac{1}{2}\norm{\left[\begin{array}{cc} d^{-1}a & d^{-1}b^*
\end{array}\right]}{} \left\|\left[\begin{array}{cc} 0 &y\\ y^*&0
\end{array}\right]\right\|\left\|
\left[\begin{array}{c} a^*d^{-1}\\ bd^{-1}
\end{array}\right]\right\|\le \frac{1}{2}\norm{y}{\infty}\pl.\]
Thus we have $x=d\cdot \tilde{y}\cdot d$ with
\[\norm{d}{2}^2\le 2+\delta\tau(e)\pl, \norm{\tilde{y}}{\infty}\le \frac{1}{2}\norm{y}{\infty}\pl.\]
Choosing $\delta$ small enough yields the assertion.
\end{proof}
We define positivity and self-adjointness on $L_1(\M,L_\infty(\N))$ and $L_1(\M,L_\infty(\N))$ as follows.
We say $\rho\in L_1(\M,L_\infty(\N))$ is positive (resp.\ self-adjoint) if there exists a positive (resp.\ self-adjoint) sequence $\rho_n\in \M_0\ten \N$ such that $\rho_n \to \rho$. For two self-adjoint operators  $\rho$ and $\si$, we say $\rho\le \si$ if $\si-\rho$ is positive. The positivity and self-adjointness in $L_\infty(\M,L_1(\N))$ are defined similarly as limits of sequences in $\M\overline{\ten}\N_0$. The next lemma shows that the $L_1(\M,L_\infty(\N))$ norm for positive elements correspond to the conditional min entropy $H_{min}$. Recall that $\rho\in L_1(\M)$ is a density operator if $\rho\ge 0$ and $\tau_\M(\rho)=1$.
\begin{lemma}\label{extend}Let $x\in L_1(\M,L_\infty(\N))$ be self-adjoint. Define
\[\la(x)=\inf \{ \la \,| \pl x \le \la \si \ten 1 \pl\text{for some density operator} \pl \si \in L_1(\M)\}\pl.\]
Then
\begin{enumerate}
\item[i)] $\la(x)\le \norm{x}{L_1(\M,L_\infty(\N))}$,
\item[ii)] $\la(x)=\norm{x}{L_1(\M,L_\infty(\N))}$ if $x$ is positive.
\end{enumerate}
\end{lemma}
\begin{proof} We first discuss the case $x\in \M_0 \ten \N$. Suppose $x=(a\ten 1) y (a^*\ten 1)$ for some self-adjoint $y\in \M\ten \N$ and $\norm{a}{2}= 1$ with $a\in \M_0$. Then $x\le \norm{y}{\infty}aa^*\ten 1$, where $aa^*\in \M_0$. Then by Lemma \ref{positive}, we have
\[\la(x)\le \norm{x}{L_\infty(\M,L_1(\N))}\pl\] for $x\in \M_0 \ten \N$. Note that   $\la(x_1+x_2)\le \la(x_1)+\la(x_2)$ and hence
\begin{align*}|\la(x_1)-\la(x_2)|\le \la(x_1-x_2)\end{align*}
For general $x$ and $\epsilon>0$, we can find a self-adjoint sequence $x_n\in\M_0\ten \N$ such that $x=\sum_{n=1}^\infty x_n$ converges absolutely and \[\sum_{n}\norm{x_n}{L_1(\M,L_\infty(\N))}\le\norm{x}{L_1(\M,L_\infty(\N))}+\epsilon\pl.\]
Then $\la(x)\le \sum_{n}\la(x_n)\le \sum_{n}\norm{x_n}{L_1(\M,L_\infty(\N))}\le \norm{x}{L_1(\M,L_\infty(\N))}+\epsilon$. Since $\epsilon$ is arbitrary, this proves i).

To prove ii), first let $x\in e\M e\ten \N$ be positive. If $x\le \la \si \ten 1$ for some density operator $\si\in \M_0$, we can choose $\tilde{\si}=\si+\delta e$ invertible in $e\M e$ with $\tau_\M(\tilde{\si})\le 1+\epsilon$. Then, we have \[0\le y=\tilde{\si}^{-\frac{1}{2}} \cdot x\cdot \tilde{\si}^{-\frac{1}{2}}\le \la 1\pl, \pl x = \tilde{\si}^{\frac{1}{2}}\cdot y \cdot \tilde{\si}^{\frac{1}{2}}\pl.\]
Hence, we obtain \begin{align}\norm{x}{L_1(\M,\L_\infty(\N))}\le \inf\{\la \pl | \pl  x\le \la \si\ten 1\pl, \pl \si\in \M_0 \text{ density operator} \} \label{hmin}\end{align}
Then it suffices to show that $\la(x)$ equals the right hand side. Suppose $x\le \la\si\ten 1$ for some density operator $\si\in L_1(\M)$. Without losing generality, we can assume that $\si$ is invertible on $e\M e$. By definition, for any positive $y\in \M
\overline{\ten} \N_0$,
\[\la\tau((\si\ten 1)y)\ge \tau(x y) \pl.\]
This implies $\norm{\si^{-\frac{1}{2}}x \si^{-\frac{1}{2}}}{}\le \la+\epsilon$. We modify $\si$ to a density operator $\tilde{\si}\in \M$ such that $\tilde{\si}=\si e_{[0,k)}+k e_{[k,\infty)}$ where $e_{[0,k]}$ is the spectral projection of $\si$ for the interval $[0,k]$.
Note that for any $z\geq 0$,
\[ (\min\{z,k\})^{-1} -z^{-1}=(z-\min\{z,k\})/z(\min\{z,k\})=\begin{cases}
                                                         0, & \mbox{if } z\le k \\
                                                         \frac{z-k}{zk}, &  \mbox{if } z>k.
                                                       \end{cases}\]
Then by functional calculus, $\norm{\tilde{\si}^{-1}-\si^{-1}}{\infty}\le \frac{1}{k}$. Therefore,
\begin{align*}\norm{\tilde{\si}^{-\frac{1}{2}}x \tilde{\si}^{-\frac{1}{2}}}{}=&\norm{x ^{\frac{1}{2}}\tilde{\si}^{-1}x ^{\frac{1}{2}}}{}
=\norm{x ^{\frac{1}{2}}\si^{-1}x ^{\frac{1}{2}}}{}
+\norm{x ^{\frac{1}{2}}(\tilde{\si}^{-1}-\si^{-1})x ^{\frac{1}{2}}}{}\le (\la+\epsilon)+\frac{1}{k}\norm{x}{\infty}.
\end{align*}
By choosing $k$ large enough, we have
\[x \le (\la+2\epsilon)\tilde{\si}\ten 1\pl.\]
where $\norm{\tilde{\si}}{\infty}\le k$ hence belongs to $\M_0$. This proves ii) for positive $x\in \M_0\ten \N$. For a general positive element $x\in L_1(\mathcal M, L_\infty(\mathcal N))$, let $x_n$ be a  sequence of positive operators in $\M_0\ten \N$ such that $\norm{x_n-x}{L_1(\M,L_\infty(\N))}\to 0$
Then by i), we know \[\la(x)=\lim_{n}\la(x_n)=\lim_n \norm{x_n}{L_1(\M,L_\infty(\N))}=\norm{x}{L_1(\M,L_\infty(\N))}\pl,\]
which completes the proof.
\end{proof}

The next lemma shows that $\la(\rho)$ is attained by the duality \[L_\infty(\M,L_1(\N)) \subset \Big(L_1(\M,L_\infty(\N))\Big)^*\pl.\]
\begin{lemma}\label{sup}Let $\rho\in L_1(\M,L_\infty(\N))$ be self-adjoint. Then
\begin{align*}\la(\rho)=\sup \{\tau(x\rho)\pl | \pl x\in \M\overline{\ten} \N_0, x\ge 0,   \norm{x}{L_\infty(\M,L_1(\N))}=1 \}.
\end{align*}
In particular, if $\rho\in L_1(\M,L_\infty(\N))$ is positive, then
\begin{align*}\norm{\rho}{L_1(\M,L_\infty(\N))}
=&\sup \{\tau(x\rho)\pl | \pl x\in \M\overline{\ten} \N_0, x\ge 0,   \norm{x}{L_\infty(\M,L_1(\N))}=1 \}.
\end{align*}
\end{lemma}
\begin{proof} By Lemma \ref{extend}, it suffices to consider $\rho\in \M_0\ten \N$. Let $\rho\in e\M e\ten \N$ for some $\tau_\M(e)<\infty$. We can assume $\M$ is finite by restricting to $e\M e$. Let us first consider the case that $\N$ is finite. We use a standard Grothendieck-Pietsch
separation argument.  Let $\la$ be a positive number such that $\la<\la(\rho)$.
We know from \eqref{hmin} that for any density operator $\si\in \M_0$,
$\la (1\ten \si)-\rho$ is not positive and hence has nontrivial negative part. Then there exists a positive $x\in L_\infty(\M
\overline{\ten} \N )$ such that $\norm{x}{\infty}=1$ and
\[ \tau(\rho x )-\la\tau((\si\ten 1) x) >0. \]
Consider the weak$^*$ compact subset
\[B=\{x\in  \M\overline{\ten} \N| \norm{x}{\infty}\le 1, x\ge 0\}.\]
For each positive operator $\si\in \M_0$ with $\tau_\M(\si)\le 1$, we define  the function $f_\si: B\to \mathbb R$ as follows (we suppress the dependence on $\rho$ since $\rho$ is fixed)
\[ f_\si(x)=\tau (\rho x)-\la \tau((\si\ten 1) x)\pl, x \in B.\]
These $f_\si$ are continuous with respect to weak$^*$ topology on $B$ because $\N$ is finite and both $\si\ten 1$ and $\rho$ are in $L_1(\M\overline{\ten} \N)$. Denote $C(B,\mathbb{R})$ as the space $w^*$-continuous real function on $B$. We define two subsets
\begin{align*}&\mathcal{F}=\{ f_\si\in C(B,\mathbb{R}) \pl | \pl \si\in \M_0, \si\ge 0, \tau_\M(\si)\le 1\pl\}\\
&\F_-=\{ f\in C(B,\mathbb{R}) \pl |\pl  \sup f< 0\}.
\end{align*}
Both $\F$ and $\F_-$ are convex sets and $\F_-$ is open. Moreover, $\F$ and $\F_-$ are disjoint because for each $f_\si\in \F$, $\sup_{x} f_\si(x)>0$. Then by the Hahn-Banach Theorem, there exists a norm-one linear function $\psi: C(B,\mathbb{R})\to \mathbb{R}$ such that for any $f_-\in\F_-$ and $f_\si\in \F$, there exists a real number $r$ such that
\begin{align*}
\phi(f_-)<r \le \phi(f_\si)\pl.
\end{align*}
Because $\F_-$ is a cone, $r\ge 0$. Similarly, $r\le 0$ because for any $0<\delta<1$, $\delta\F\subset \F$. Then $r=0$ and $\phi$ is a positive linear functional because $\phi(f_-)< 0 $ for any $f_-\in \F_-$. By the Riesz Representation Theorem, $\phi$ is given by a Borel probablity measure $\mu$ on $B$. Namely,
\[\phi(f)=\int_B f(x)\mu(x)\pl.\]
Denote $x_0=\int_B x d\mu(x)$. We have for any positive operator $\si\in \M_0$ with  $\tau_\M(\si)\le 1$, that
\begin{align*} \phi(f_\si)= \int_B f(x)d\mu(x)&=\int_B \tau(\rho x)-\la \tau((\si\ten 1) x)d\mu(x)
=\tau(\rho x_0)-\la \tau((\si\ten 1) x_0)\ge 0.
\end{align*}
By Lemma \ref{positive},
\[ \tau ( \rho x_0)\ge \la\sup \{ \tau((\si\ten 1) x_0)\pl |  \si\in \M_0 , \pl  \tau_\M(\si_0)\le 1  , \si\ge 0 \}=\la \norm{x_0}{L_\infty(\M,L_1(\N))}.\]
Normalizing $\tilde{x}_0=\norm{x_0}{L_\infty(\M,L_1(\N ))}^{-1}x_0$, we have $\tau(\rho\tilde{x}_0)\ge \la$. This proves the case for finite $\N$. For semifinite $\N$, we define for each projection $p\in \N$ with $\tau_\N(p)<\infty$,
\[\la_p=\inf\{\la \pl | \pl (1\ten p)\rho(1\ten p)\le \la \si \ten p \pl\text{for some density operator}\pl \si\in \M_0  \}\pl.\]
For two projections $p_1\le p_2$,  we have $\la_{p_1}\le \la_{p_2}$. Thus $\la_p$ is monotone non-decreasing over $p$ for the natural ordering. Based on the finite case, it suffices to show that
$\lim_{p} \la_p\ge \la(\rho)$. Write $\la_1=\lim_{p} \la_p$. Given $\epsilon>0$, for $p$ large enough there exists a density operator $\si_p\in \M_0$ such that
\begin{align*}(1\ten p)\rho(1\ten p)\le (\la_1+\epsilon) \si_p \ten p\pl.\end{align*}
Let $\psi_{p}:\M\overline{\ten} \N \to \mathbb{C}$ be the positive normal linear functional $\psi_p(x)=\tau((1\ten p)\rho(1\ten p)x)$ and let $\psi$ be the normal weight $\psi(x)=\tau(\rho x)$. Let $\xi_p:\M\to \mathbb{C}$
be the normal state $\xi_p(y)=\tau_\M(\si_p y)$. Let $\xi$ be a weak$^*$-limit point of $\xi_p$ in $\M^*$. Then $\xi$ is a state on $\M$ and it decomposes into a normal part and a singular part $\xi=\xi_n+\xi_s$. For any positive $x\in (\M\ten \N_0)_+$
\[\psi_p(x)\le (\la_1+\epsilon)\lim_{p} \xi_p\ten \tau_\N(x)=(\la_1+\epsilon)\xi\ten \tau_\N(x)\pl.\]
By the normality of $\psi_p$, we have $\psi_p\le (\la_1+\epsilon)\xi\ten \tau_\N$ as normal states on $\M\overline{\ten}p\N p$.
For any positive $x\in (\M\ten \N_0)_+$, there exists a $p_x$ such that for $p\ge p_x$, $\psi(x)=\psi_{p}(x)$ and hence
\[\psi(x)=\psi_{p}(x)\le (\la_1+\epsilon)\xi_n\ten \tau_\N(x)\pl.\]
By normality, $\psi\le (\la_1+\epsilon)\xi_n\ten \tau_\N$ as weights.
Since $\xi_n$ is a sub-state (that is, a positive linear functional with norm $\leq 1$), $\xi_n(y)=\tau_\M(y\si)$ for some positive $\si\in L_1(\M)$ with $\tau_\M(\si)\le 1$. Then we have
\[\rho \le (\la_1+\epsilon)\si\ten 1\pl.\]Since $\epsilon$ is arbitrary, we complete the proof.
\end{proof}
This next lemma is an analogue  of the Choi matrix.
\begin{lemma}\label{map}
There is a contraction\begin{align*} &L_\infty(\M,L_1(\N))\longrightarrow CB(L_1(\M^{op}),L_1(\N))\pl,\\
& x \mapsto T_x\in CB(L_1(\M),L_1(\N))\pl , \pl\pl T_x(\rho^{op})=\tau_\M\ten \operatorname{id}_\N((\rho\ten 1)x).
\end{align*} Moreover, \begin{enumerate}
\item[i)]for any positive $x$, $\norm{x}{L_\infty(\M,L_1(\N))}=\norm{T_x}{cb}$.
\item[ii)]$T_x$ is completely positive if and only if $x$ is positive.
\item[iii)]$T_x$ is trace preserving if and only if $\operatorname{id}\ten\tau_\N(x)=1_{\M}$.
\item[iv)] for $S\in CB(L_1(\N),L_1(\N))$, $S\circ T_x=T_{\operatorname{id}\ten S(x)}$.
\item[v)] for any finite rank $T:L_1(\M^{op})\to L_1(\N)$, $T=T_x$ for some $x\in \M\ten L_1(\N)$.
\end{enumerate}
\end{lemma}
\begin{proof}By a density argument, it suffices to discuss $x\in \M\overline{\ten} p\N p$ with $\tau_\N(p)<\infty$. Given $\rho \in L_1(\M )$,
$(\rho\ten 1_\N) x=(\rho\ten p) x \in L_1(\M\overline{\ten} p\N p)$ hence the map $T_x(\rho)=\tau_\M\ten \operatorname{id}_\N((\rho\ten 1_\N)x)\in L_1(\N)$ is well defined. For $\norm{\rho^{op}}{L_1(\M^{op})}=1$, we have $\rho=ba$ for some ${\norm{a}{2}=\norm{b}{2}=1}$. Note that
$\tau_\M\ten \operatorname{id}_\N((ba\ten 1_\N)x)= \tau_\M\ten \operatorname{id}_\N((a\ten 1)  x( b\ten 1))$. Then
\[ \norm{T_x(\rho^{op})}{L_1(\N)}\le \norm{a\cdot x \cdot b}{L_1(\M\overline{\ten}\N)}\le \norm{x}{L_\infty(\M,L_1(\N))}\pl.\]
 Let $e_{ij}$ be the matrix units in $M_n$ and $S^n_2$ be the Schatten $2$-class. For the completely bounded norm, we first note that
$\operatorname{id}_{M_n}\ten T_x= T_{\phi\ten x}: L_1(M_n(\M)^{op})\to L_1(M_n(\N))$ where $\phi=\sum_{i,j} e_{ij}\ten e_{ij}\in M_n\ten M_n$ and $\phi\ten x\in L_\infty( M_n\ten \M, L_1(M_n(\N)))$. Here $\phi$ is the Choi matrix for $\operatorname{id}:M_n\to M_n$. Given $\norm{a}{S_2^n(L_2(\M))}=\norm{b}{S_2^n(L_2(\M))}=1$, we can write
$a=\sum_{k}\mu_k\omega_k\ten a_k$ such that $\mu_k$ (resp. $a_k$) orthogonal in $S_2^n$ (resp. $L_l(\M)$) and $\norm{a_k}{2}=1, \sum_{k}\norm{\omega_k}{2}^2=1$ and similarly for $b=\sum_{l}\nu_l\si_l\ten b_l$. Then
\begin{align*}\operatorname{id}_{M_n}\ten T_x(ab)=T_{\phi\ten x}(ab)&=\sum_{k,l}\Big(\tr\ten \operatorname{id}_{M_n}((\omega_k\si_l\ten 1)\phi)\Big)\ten \Big(\tau_\M\ten \id_\N ((a_kb_l\ten 1) x)\Big)\\
&=\sum_{k,l}\Big(\tr\ten \operatorname{id}_{M_n}(\omega_k\cdot \phi \cdot \si_l)\Big)\ten \Big(\tau_\M\ten \id_\N (a_k\cdot x \cdot b_l)\Big)\\
&=\tr\ten id_{M_n}\ten\tau_\M\ten \id_\N \Big(\sum_{k,l} (\omega_k\cdot \phi \cdot \si_l)\ten (a_k\cdot x \cdot b_l)\Big)\pl.
\end{align*}
Using bracket notation,
\[\phi=\ket{h}\bra{h}\pl, \ket{h}=\sum_{i=1}\ket{i}\ket{i}\]
where $\{\ket{i}\}$ is the standard basis in $l_2^n$. We have
\[\norm{\omega_k\cdot \phi \cdot \si_l}{1}=\norm{\omega_k\ten 1 \ket{h}}{l_2}\norm{\si_l^*\ten 1 \ket{h}}{l_2}
=\norm{\omega_k}{2}\norm{\si_l}{2}
\pl.\]
Here $\norm{\cdot}{l_2}$ is the vector norm and \[\norm{\omega_k\ten 1 \ket{h}}{l_2}^2=\bra{k}\omega_k^*\omega_k\ten 1 \ket{h}=\tr(\omega_k^*\omega_k)=\norm{\omega_k}{2}\pl.\]
Therefore,
\begin{align*}\norm{\sum_{k,l} (\omega_k\cdot \phi \cdot \si_l)\ten (a_k\cdot x \cdot b_l)}{1} &\le \sum_{k,l}\norm{\omega_k\cdot \phi \cdot \si_l}{1}\norm{a_k\cdot x \cdot b_l}{1}
\\ &\le \sum_{k,l} \norm{\omega_k}{2}\norm{\si_l}{2} \norm{x}{L_\infty(\M,L_1(\N))}
\le \norm{x}{L_\infty(\M,L_1(\N))}
\end{align*}
By $\norm{\operatorname{id}_{M_n}\ten T_x(ab)}{1}\le \norm{\sum_{k,l} (\omega_k\cdot \phi \cdot \si_l)\ten (a_k\cdot x \cdot b_l)}{1}$, this implies \[\norm{\operatorname{id}_{M_n}\ten T_x:L_1(M_n(\M)^{op})\to L_1(M_n(\N))}\le \norm{x}{L_\infty(\M,L_1(\N))}\pl. \]
Then by $ L_1(M_n(\M)^{op}, \tr\ten \tau_\M)\cong S_1^n(L_1(\M^{op}))$ and \cite[Lemma 1.2]{pisier93}, we obtain
\begin{align*}\norm{T_x:L_1(\M^{op})\to L_1(\N)}{cb}&=\sup_n\norm{\operatorname{id}_n\ten T_x: S_1^n(L_1(\M^{op}))\to  S_1^n(L_1(\N))}{}
\\&\le \norm{\phi\ten x}{{L_\infty(M_n(\M), L_1(M_n(\N)))}}=\norm{x}{{L_\infty(\M, L_1(\N))}}
\pl.
\end{align*}
Now suppose $x$ is positive. For a density operator $\rho\in L_1(\M^{op})$,
\[T_x(\rho^{op})=\tau_\M\ten \operatorname{id}_\N((\rho\ten 1) x)=\tau_\M\ten \operatorname{id}_\N(\rho^{\frac{1}{2}}\cdot x\cdot  \rho^{\frac{1}{2}})\ge 0\]
Applying the same argument for $\phi\ten x$, we know $T_x$ is completely positive. Then taking the supremum over all density operators $\rho$,
\[\sup_{\rho}\norm{T_x(\rho^{op})}{1}=\sup_{\rho}\tau_\M\ten \tau_\N(\rho^{\frac{1}{2}}\cdot x \cdot \rho^{\frac{1}{2}})=\norm{x}{L_\infty(\M,L_1(\N))}\pl.\]
Thus for positive $x$, we find $\norm{T_x}{cb}=\norm{x}{L_\infty(\M,L_1(\N))}=\norm{T_x}{}\le \norm{T_x}{cb}$, which proves i). For ii), we note that the ``if'' statement follows by the construction of $T_x$. To prove the ``only if'' statement, we conversely suppose $x$ is not positive and  we show that  $T_x$ is not completely positive. There exists a vector ${h}=\sum_{j=1}^n a_j\ten b_j\in L_2(\M)\ten_2 L_2(\N)$ such that $a_j\in\M_0, b_j\in \N_0$, $\langle h, x{h}\rangle\ngeq 0$ (that is, the inner product is either not real or is negative). This means
\begin{align*}
\langle h, x{h}\rangle &=\sum_{i,j=1}^n\tau_\M\ten\tau_\N( (a_i^*\ten b_i^*)x (a_j\ten b_j))\\
&= \tau_\N\Big(\sum_{i,j=1}^n b_i^*\tau_\M\ten \operatorname{id}_\N\big( (a_i^*\ten 1)x (a_j\ten 1)\big)b_j\Big)\ngeq 0\pl.
\end{align*}
Thus, $\Big(\tau_\M\ten \operatorname{id}_\N\big( (a_i^*\ten 1)x (a_j\ten 1)\big)\Big)_{i,j=1}^n$ is not positive in $S_1^n(L_1(\N))$.
Note that $\omega^{op}=\sum_{i,j=1}^{n}e_{ij}\ten (a_i^*)^{op}(a_j)^{op}=\sum_{i,j=1}^{n}e_{ij}\ten (a_ja_i^*)^{op}$ is positive in $S_1^n(L_1(\M^{op}))$. Then $T_x$ is not completely positive because
\begin{align*}id_n\ten T_x(\omega)=&\sum_{i,j=1}^n e_{ij}\ten \Big(\tau_\M\ten \operatorname{id}_\N( (a_ja_i^*\ten 1)x)\Big) \\=&\sum_{i,j}e_{ij}\ten\Big(\tau_\M\ten \operatorname{id}_\N( (a_i^*\ten 1)x (a_j\ten 1))\Big)\ngeq 0\pl.\end{align*}
This proves ii). For any $\rho\in L_1(\M)$, by Fubini's theorem,
\begin{align*} \tau_\N (T_x(\rho^{op}))=&\tau_\N\Big(\tau_\M\ten \operatorname{id}_\N  ((\rho\ten 1) x)\Big)=\tau_\M \Big(\rho \operatorname{id}_\M \ten\tau_\N(x)\Big) \pl.\end{align*}
Thus $T_x$ is trace preserving if and only $\operatorname{id}_\M \ten\tau_\N(x)=1$. This verifies iii). For iv),
let $S\in CB(L_1(\N), L_1(\N))$. For $ \rho\in L_1(\M)$,
\begin{align*}S\circ T_x(\rho^{op})=&S\Big(\tau_\M\ten \operatorname{id}_\N  ((\rho\ten 1) x)\Big)
 \\&=\tau_\M\ten \operatorname{id}_\N \Big((\rho\ten 1) \operatorname{id}\ten S(x) \Big)
\\& = T_{\operatorname{id}\ten S(x)}(\rho^{op})\pl.
\end{align*}
Finally, for v), let $T$ be a finite rank map from $L_1(\M^{op})$ to $L_1(\N)$. Then there exists finite $y_j\in \M$ and $z_j\in L_1(\N)$ such that $T(\rho^{op})=\sum_{j=1}^n\tau_\M(\rho y_j)z_j$. Then $T=T_x$ for $x=\sum_{j=1}^n y_j\ten z_j$ which belongs to $\M\ten L_1(\N)$. That  completes the proof.
\end{proof}

The above lemma gives a contraction
\[L_\infty(\M,L_1(\N))\to CB(L_1(\M^{op}),L_1(\N))\subset CB(\N^{op},\M)\pl.\]
Note that $CB(\N^{op},\M)_*=L_1(\M^{op})\widehat{\ten}L_\infty(\N^{op})$. The pairings for an algebraic tensor $\rho^{op}=\sum_{j=1}^n y_j^{op}\ten z_j^{op}\in L_1(\M)^{op}\ten \N^{op}$ to $L_\infty(\M,L_1(\N))$ and to $CB(\N^{op},\M)$ coincide,
\begin{align*}
\lan x,\rho^{op} \ran_{(L_\infty(\M,L_1(\N)), L_1(\M^{op},L_\infty(\N^{op}))}=&\tau_\M\ten\tau_\N (x\sum_{j=1}^n y_j\ten z_j)
\\=&\tau_\N\Big(\sum_{j} z_j\tau_\M\ten \operatorname{id}((y_j\ten 1)x)\Big)
\\=&\tau_\N\Big(\sum_{j} z_jT_x(y_j)\Big)=\tau_\N\Big(\sum_{j} T_x^\dag(z_j) y_j\Big)
\\=&\lan T_x, \rho^{op}\ran_{(CB(\N^{op},\M), L_1(\M^{op})\widehat{\ten}L_\infty(\N^{op}))}.
\end{align*}
Then, for an algebraic tensor $x=\sum_{j=1}^n y_j\ten z_j\in L_1(\M)\ten \N$, we have
\[\norm{x}{L_1(\M,L_\infty(\N))}\le \norm{x}{L_1(\M)\widehat{\ten}\N}.\]
It was proved in \cite[Theorem 3.4]{pisier93} that for hyperfinite $\M$ (i.e. $\M=\overline{(\cup_\al \M_\al )^{w^*}}$, where the union is of an increasing net of finite-dimensional subalgebras $\M_\al$, we have the isometric isomorphism
\begin{align}L_1(\M,L_\infty(\N))\cong L_1(\M)\widehat{\ten}\N.\label{eq}\end{align}
We shall show that this isomorphism is characterized by the injectivity of $\M$. Recall that a von Neumann algebra $\M$ is \emph{injective} if  there exists an embedding $\M\subset B(H)$ and a completely positive projection $P:B(H)\to \M$ with $\norm{P}{}=1$. An equivalent condition is the weak$^*$ completely positive approximation property ( \emph{weak$^*$-CPAP}). A von Neumann algebra $\M$ has weak$^*$-CPAP if there exists a net of normal finite rank completely positive maps $\Phi_\al$ such that for any $x\in \M$, $\Phi_\al(x)\to x$ in the weak$^*$ topology. In general, hyperfinite implies  injective. The converse (say, when $\M\subset B(H)$ on a separable Hilbert space $H$) is a celebrated result of Connes \cite{connes76}. We refer to \cite{pisieros} for more information about these properties.

The next theorem is a dual form of Haagerup's characterization of injectivity by decomposability \cite{haagerup84}. It suggests that the conditional min entropy connects to the projective tensor norm if and only if $\M$ is injective.
\begin{theorem}\label{injective}Let $\M,\N$ be semi-finite von Neumann algebras. Suppose $\N$ is infinite dimensional. The following are equivalent \begin{enumerate}
\item[i)] $\M$ is injective.
\item[ii)]
$L_1(\M,L_\infty(\N))\cong L_1(\M)\widehat{\ten}\N$ isomorphically
\item[iii)]
$L_1(\M,L_\infty(\N))\cong L_1(\M)\widehat{\ten}\N$ isometrically
\end{enumerate}In particular, $L_1(\M,L_\infty(\M^{op}))\cong L_1(\M)\widehat{\ten}\M^{op}$ if and only if $\M$ is injective.
\end{theorem}
\begin{proof}
We first prove i) $\Rightarrow$ iii). Suppose $L_1(\M,L_\infty(\N))\neq L_1(\M)\widehat{\ten}\N$ isometrically.
Because both spaces are norm completions of the algebraic tensor $L_1(\M)\ten \N$,  there exists $\rho=\sum_{j=1}^n y_j\ten z_j$ such that
\[ \norm{\rho}{L_1(\M,L_\infty(\N))}<1=\norm{\rho}{L_1(\M)\widehat{\ten}\N}. \]
Then by the duality $(L_1(\M)\widehat{\ten}\N)^*=CB(\N,\M^{op})$, there exists a CB map $S\in CB(\N,\M^{op})$ with $\norm{S}{cb}=1$ such that
\[1=\lan S, \rho\ran=\lan \operatorname{id}, \operatorname{id}_\M\ten S(\rho)\ran\pl. \]
Here we have \[\norm{\operatorname{id}_\M\ten S(\rho)}{L_1(\M,L_\infty(\N))}\le \norm{S}{cb}\norm{ \rho}{L_1(\M,L_\infty(\N))}<1\pl.\]
If $\M$ is injective, then there exists a net of  finite-rank, normal, unital, completely positive maps $\Phi_\al$ approximating the identity map $\operatorname{id}_\M$ in the point-weak$^*$ topology. By Lemma \ref{map}, $\Phi_\al=T_{x_\al}$ for some $x_\al\in \M^{op}\ten L_1(\N^{op})$ with
\[\norm{x_\al}{L_\infty(\M^{op},L_1(\N^{op}))}=\norm{\Phi_\al}{cb}=1\pl.\]
This leads to a contraction:
\begin{align*}1=\lan \operatorname{id}, \operatorname{id}_\M\ten S(\rho)\ran=&\lim_{\al} \lan T_{x_\al}, \operatorname{id}_\M\ten S(\rho)\ran\\=&\lim_{\al} \lan x_\al, \operatorname{id}_\M\ten S(\rho)\ran
\\\le &\lim_{\al} \norm{ x_\al}{L_\infty(\M^{op},L_1(\N^{op}))} \norm{\operatorname{id}_\M\ten S(\rho)}{L_1(\M,L_\infty(\N))}
\\\le & \norm{\operatorname{id}_\M\ten S(\rho)}{L_1(\M,L_\infty(\N))}<1.
\end{align*}
For ii)$\Rightarrow$ i), we first reduce the semi-finite $\M$ to the  finite case. We have the decomposition $\M=\oplus_{i\in i}(\M_i\overline{\ten}B(H_i))$ (see \cite[Chapter 5, Proposition ]{takesaki}) where $\M_i$ are finite von Neumann algebras and $H_i$ are Hilbert spaces.
For each $\M_i$, there exists a trace preserving embedding $\iota:\M_i\to \M$ and a projection $P:\M\to e_i\M e_i$ for some projection $e_i$ such that $P\circ\iota= \id_{\M_i}$. This induces the isometric  embedding
\[L_1(\M_i,L_\infty(\N))\subset L_1(\M,L_\infty(\N))\pl, L_1(\M_i)\hten \N\subset L_1(\M)\hten \N\pl.\]
Suppose $L_1(\M,L_\infty(\N))\cong L_1(\M)\widehat{\ten}\N$ isometrically. We have for each $i$, $L_1(\M_i,L_\infty(\N))\cong L_1(\M_i)\widehat{\ten}\N$ isometrically. It suffices to show that this implies $\M_i$ is injective.

We now assume $\M=\M_i$ finite. Let $l_\infty^n$ be the $n$-dimensional commutative $C^*$-algebra. Because $\N$ is infinite dimensional, for any $n$ there exists completely positive and contractive maps (see \cite[Lemma 2.7]{haagerup84})
\[Q:l_\infty^n\to \N, R:\N\to l_\infty^n\]
such that $R\circ Q=\id_{l_\infty^n}$. Both
$\id_\M\ten R$ and $\id_\M\ten Q$ extend to complete contractions
\begin{align*}L_1(\M)\hten l_\infty^n\overset{\id_\M\ten R}{\longrightarrow}  L_1(\M)\hten \N\overset{\id_\M\ten Q}{\longrightarrow}  L_1(\M)\hten l_\infty^n\pl, \\
L_1(\M,l_\infty^n)\overset{\id_\M\ten R}{\longrightarrow}  L_1(\M,L_\infty(\N))\overset{\id_\M\ten Q}{\longrightarrow} L_1(\M,l_\infty^n)\pl. \\
\end{align*}
Thus we have the isometric  imbeddings
\begin{align*}
 L_1(\M,l_\infty^n)\subset L_1(\M,L_\infty(\N))\textnormal{ and }  L_1(\M)\widehat{\ten}l_\infty^n \subset L_1(\M)\widehat{\ten}\N.
\end{align*}
Suppose $L_1(\M,L_\infty(\N))\cong L_1(\M)\widehat{\ten}\N$ isomorphically. Then we have $L_1(\M,l_\infty^n)\cong L_1(\M)\widehat{\ten}l_\infty^n$ for each $n$, and moreover a uniform constant   $c$ such that for all $n$, \[c\norm{\rho} {L_1(\M)\widehat{\ten}l_\infty^n}\le \norm{\rho} {L_1(\M)\widehat{\ten}l_\infty^n}\norm{\rho}{L_1(\M,l_\infty^n)}\le \norm{\rho} {L_1(\M)\widehat{\ten}l_\infty^n}\pl.\]
At the dual level, for each $T:l_\infty^n\to \M^{op}$,
\begin{align}\norm{T}{cb}=\norm{T_x}{cb}\le \norm{x}{L_\infty(\M^{op},l_1^n)}\le c^{-1}\norm{T_x}{cb}\pl. \label{equality}\end{align}
Here $T=T_x$ as in Lemma \ref{map}, for $x=\sum_{j=1}^nT(e_j)\ten e_j\in \M^{op}\ten l_1^n$ with $e_j\in l_1^n$ being the dual standard basis of $l_\infty^n$.
We shall suppress the ``op'' notation since it is equivalent to consider $\M$ and $\M^{op}$ here. For any $n$ unitaries $u_j$ and a central projection $q$ in $\M$, we consider $x_u=q\sum_{j=1}^n u_j\ten e_j$. We have
\begin{align*}\norm{x_u}{L_\infty(\M,l_1^n)}&=\sup\{\norm{q\sum_{j=1}^n au_jb\ten e_j}{L_1(\M,l_1^\infty)} \pl|\pl \norm{a}{L_2(\M)}=\norm{b}{L_2(\M)}=1\}\\
&=\sup\{\sum_{j}\norm{qau_jb}{L_1(\M)} \pl|\pl \norm{a}{L_2(\M)}=\norm{b}{L_2(\M)}=1\}
\\ &\ge \sum_{j}\tau_\M(q)^{-1}\norm{qu_j}{L_1(\M)}
\\ &= \sum_{j=1}^n1 =n.
\end{align*}
Here we have chosen $a=b=\tau_\M(q)^{-1/2}q$. Then by \eqref{equality}, we have \[\norm{T_u}{cb}\ge c\norm{x_u}{L_\infty(\M,l_1^n)}=cn\pl,
T_{u}:l_\infty^n \to \M\pl,\pl\pl \pl T_u((c_j)_j)=q\sum_{j=1}^n c_ju_j\pl.\]
Then it follows from \cite[Lemma 2.3 \& Lemma 2.5]{haagerup84} that $\M$ is injective. Since iii)$\Rightarrow$ ii) is trivial, this  completes the proof.
\end{proof}

\subsection{Quantum Majorization} \label{QM-VNA}
We now discuss quantum majorization for semifinite von Neumann algebras. We will focus on the case where $\mathcal M$ is injective, because by Theorem \ref{injective}, beyond injectivity we lose the duality between $H_{min}$ entropy and CPTP maps.

We say $T:L_1(\M)\to L_1(\M)$ is completely positive trace preserving (resp. trace non-increasing) if its adjoint $T^\dag:\M^{op}\to \M^{op}$ is normal completely positive and unital (resp. sub-unital). We will use the abbreviation CPTP for completely positive trace preserving, CPTNI for completely positive trace non-increasing and UCP for unital completely positive. We start with a consequence of Lemma \ref{sup} and Theorem \ref{injective}.

\begin{prop}\label{attain}Let $\M$ be injective.\begin{enumerate}
\item[i)]For a self-adjoint $x\in L_1(\M)\hten \N$,
\begin{align*}\la(x)=&\sup \{\lan \Phi, x\ran| \Phi:L_1(\M)\to L_1(\N^{op})\pl \text{CPTNI}\} \end{align*}
\item[ii)]Define the real part of $x\in L_1(\M)\hten \N$ as $\Re \pl  x=(x+x^*)/2$. Then
\begin{align*}\la(\Re \pl x)=&\sup \{\Re\pl \lan \Phi, x\ran| \Phi:L_1(\M)\to L_1(\N^{op})\pl \text{CPTNI}\} \end{align*}
\item[iii)]For positive $\rho$,
\begin{align*}\norm{\rho}{L_1(\M)\hten\N}=&\sup \{\lan \Phi, \rho\ran| \Phi:L_1(\M)\to L_1(\N^{op})\pl \text{CPTNI}\}\\
=&\sup \{\lan \Phi, \rho\ran| \Phi:L_1(\M)\to L_1(\N^{op})\pl \text{CPTP}\} \end{align*}
\end{enumerate}
\end{prop}
\begin{proof}We first show that $\lan \Phi, y\ran\ge 0$ for a positive $y\in L_1(\M)\hten \N$ and CP $T:\N\to \M^{op}$. By density argument, it suffices to consider $\rho\in \M_0\ten \N$. Suppose $y=(\sum_{j=1}^n a_j\ten b_j)^*(\sum_{j=1}^n a_j\ten b_j)$ for some $a_j\in e\M e$ and $b_j\in \N$. Then
${\operatorname{id}_n\ten T(\sum_{i,j=1}^n e_{ij}\ten b_i^*b_j)}=\sum_{i,j=1}^n e_{ij}\ten T(b_i^*b_j)$ is positive in $M_n(\M^{op})$.
Therefore,
\begin{align*}
\lan T,y\ran= \tau_{\M}(\sum_{i,j=1}^{n} (a_i^*a_j)^{op} T(b_i^*b_j))=\sum_{i,j=1}^{n}\tau_{\M^{op}}( (a_i^{op})^* T(b_i^*b_j)a_j^{op})= \sum_{i,j=1}^{n} \bra{a_i^{op}} T(b_i^*b_j) \ket{a_j^{op}}\ge 0,
\end{align*}
where $\ket{a_j^{op}}\in L_2(\M^{op},\tau_\M)$ is the vector of $a_j^{op}$ in the GNS representation. Thus,
$\lan T,y\ran\ge 0$ for CP $T:\N\to \M^{op}$ and also
CP $T:L_1(\M)\to L_1(\N^{op})$ as   normal maps.
Then if $x\le \la 1\ten \si$ and $\Phi: L_1(\M)\to L_1(\N^{op})$ CPTNI, we have
\[\lan\Phi , x\ran\le  \la \lan\Phi , 1\ten \si\ran=\la\tau_\N(\Phi(\si))\le \la \pl, \]
which implies $\lan\Phi , \rho\ran\le \la(\rho)$. On the other hand, by Theorem \ref{injective} and Lemma \ref{sup},
\begin{align*}
\la(x)&=\sup \{\tau(xy)\pl |\pl y\in \M\overline{\ten}\N_0 ,y\ge 0,  \norm{y}{L_1(\M,L_\infty( \N^{op}))}= 1\}
\\&=\sup \{\lan T_y, \rho\ran\pl |\pl y\in \M\overline{\ten}\N_0 ,y\ge 0,  \norm{y}{L_1(\M,L_\infty( \N^{op}))}= 1\}
\\&\le \sup \{\lan T, x\ran\pl | T:L_1(\M)\to L_1(\N) \pl \text{CPTNI}\pl\}\\ &\le \la(x) \pl.
\end{align*}
This proves i). ii) follows from the fact that for any CP $T$, $\Re \lan T,x\ran =\lan T, \Re\pl x\ran$. For iii),
given a CPTNI map $T$, one can always find a CPTP $\tilde{T} $ such that $\tilde{T}-T$ is CP. Therefore,
\begin{align*} &\norm{\rho}{L_1(\M)\hten\N}=\la(\rho)=\sup_{T \pl \text{CPTNI}} \lan T, \rho\ran\le \sup_{T \pl \text{CPTP}} \lan T, \rho\ran\le \la(\rho)\pl. \qedhere
\end{align*}
\end{proof}


\begin{lemma}
\label{convex} Let $\rho$ be a bipartite density operator in $L_1(\M\overline{\ten} \N)$. The set
\[C(\rho)=\{\Phi\ten \operatorname{id}(\rho)\pl |\pl  \Phi:L_1(\M)\to L_1(\M) \pl \text{CPTP}\}\]
 is a closed set in $L_1(\M\overline{\ten}\N)$ with respect to the topology induced by \[\M^{op}\ten_{\min}\N^{op}\subset \M^{op}\overline{\ten}\N^{op}=L_1(\M\overline{\ten}\N)^*.\]
 In particular, $C(\rho)$ is a norm closed set in $L_1(\M\overline{\ten}\N)$.
 \end{lemma}
\begin{proof}Let $\si\in L_1(\M\overline{\ten}\N)$ and $\Phi_\al$ be a net of CPTP maps such that $\Phi_\al\ten \operatorname{id}(\rho)\to \si$ with respect to $\M^{op}\ten_{\min}\N^{op}$. That is, for any $x\in \M\ten \N$
\begin{align}\lim_{\al}\tau(x \Phi_\al\ten \operatorname{id}(\rho))=\tau(x \si)\pl. \label{converge}\end{align}
Taking $x=1_\M\ten 1_\N$, this implies $\tau(\si)=\lim_{\al}\tau( \Phi_\al\ten \operatorname{id}(\rho))=1$. Note that the \[CB(L_1(\M),L_1(\M))\subset CB(\M^{op},\M^{op})=(L_1(\M)\hten\M^{op})^*\pl.\]
By weak$^*$-compactness, there exists a sub-net $\Phi_{\beta}$ such that their corresponding subnet of adjoints $\Phi_{\beta}^{\dag}:\M^{op} \to \M^{op}$ converges to some $\Phi^\dag:\M^{op} \to \M^{op}$ in the point-weak$^*$ topology. That is, for $x\in L_1(\M), y^{op}\in\M^{op}$, we have
\[\lim_{\beta}\tau_\M\big(x\Phi_{\beta}^\dag(y^{op})\big)=\tau_\M\big(x\Phi^\dag(y^{op}) \big) \pl.\]
Then it is clear that $\Phi^\dag$ is UCP. Note that
$(\M^{op})^*= L_1(\M)\oplus L_1(\M)^{\perp}$ decomposes into a normal part and a singular part. Let $\Phi:L_1(\M)\to (\M^{op})^*$ be the restriction of the double adjoint map $\Phi^{\dag\dag}:(\M^{op})^*\to (\M^{op})^*$. Then $\Phi_{\beta}\ten \id(\rho)\to \Phi\ten \id(\rho)$ in the sense that for any $x\in \M\ten\N$
\[ \tau(x\Phi_{\beta}\ten \id(\rho))=\tau(\Phi_{\beta}^\dag\ten \id(x)\rho)\to \Phi\ten \id(\rho)(x)\pl,\]
where $\Phi\ten \id(\rho)\in (\M^{op})^*\hten L_1(\N)$. Then by \eqref{converge},  for any $x\in \M\ten \N$,
\[\Phi\ten \id(\rho)(x)=\tau(\si x):=\si(x)\]
where the density operator $\si$ is viewed as a normal state. Decompose the map
$\Phi=\Phi_n+\Phi_s$ where $\Phi_n\in CB(L_1(\M),L_1(\M))$ is the normal part and $\Phi_s\in CB(L_1(\M), L_1(\M)^{\perp})$ is the singular map. Then for any $x\in \M\ten\N$,
\begin{align}\Big(\si-\Phi_n\ten \id(\rho)\Big)(x)=\Phi_s\ten \id(\rho)(x)\pl \label{coin}\end{align}
where $\si-\Phi_n\ten \id(\rho)\in L_1(\M)\hten L_1(\N)$  and $\Phi_s\ten \id(\rho)\in (\M^{op})^*\hten  L_1(\N)$. Let $\omega_1,\omega_2:\M\to \mathbb{C}$ be the linear functionals defined by
\begin{align*}&\omega_1(y):=\Big(\si-\Phi_n\ten \id(\rho)\Big)(y\ten 1)\pl, \pl
\omega_2(y):=\Phi_s\ten \id(\rho)(y\ten 1)\pl, \pl y\in \M.
\end{align*}
Then $\omega_1$ is normal and $\omega_2$ is singular. By \eqref{coin}, $\omega_1=\omega_2$ which implies $\omega_1=\omega_2=0$. Therefore,
\[\Phi_s\ten \id(\rho)(1\ten 1)=\omega_2(1)=0\pl.\]
Hence $\Phi_s\ten \id(\rho)=0$. We have $\si=\Phi_n\ten \id(\rho)$ for $\Phi_n:L_1(\M)\to L_1(\M)$ CPTNI. Define $\Phi_0(x)=\tau(\Phi_n(x)-x)\omega$ for any density operator $\omega\in L_1(\M)$. Then $\tilde{\Phi}=\Phi_n+\Phi_0$ is a CPTP map and $\tilde{\Phi}\ten id(\rho)=\si$. This completes the proof.
\end{proof}
We say a CPTP map $\Phi:L_1(\M)\to L_1(\N)$ is \emph{entanglement-breaking} if $\Phi(\rho)=\sum_{j=1}\tau(x_j\rho)\omega_j$ for some set of $x_j$, $j=1, 2, \dots$, satisfying $\sum_{j=1}x_j=1$ and $x_j\ge 0$ (such a set $\{x_j\}$ is called a measurement in quantum mechanics) and density operators $\omega_j$. Such a CPTP map is a  quantum channel that admits a factorization through $l_1^\infty$, which is the state space of a classical system. We now prove our main theorem with respect to quantum majorization for injective semifinite von Neumanna algebra.
\begin{theorem}\label{von} Let $\M$ and $\N$ be two semifinite von Neumann algebras and let $\M$ be injective. Let $\rho,\si$ be two density operators in $L_1(\M\overline{\ten} \N)$.
The following are equivalent:
\begin{enumerate}
\item[i)] there exists a CPTP map $\Phi:L_1(\M)\to L_1(\M)$ such that $\Phi\ten \operatorname{id}(\rho)= \si$
\item[ii)] for any CP and CB $\Psi:L_1(\N)\to \M^{op}$,
    \[ \norm{\operatorname{id}\ten \Psi(\rho)}{L_1(\M)\hten (e\M e)^{op}}\ge \norm{\operatorname{id}\ten \Psi(\si)}{L_1(\M)\hten (e\M e)^{op}}\]
\item[iii)]  for any projection $e\in \M$ with $\tau_\M(e)<\infty$ and for any entanglement-breaking CPTP map $\Psi:L_1(\N)\to L_1(e\M e^{op})\cap e\M e^{op}$,
    \[ \norm{\operatorname{id}\ten \Psi(\rho)}{L_1(\M)\hten e\M e^{op}}\ge \norm{\operatorname{id}\ten \Psi(\si)}{L_1(\M)\hten e\M e^{op}}\]
\end{enumerate}
\end{theorem}
\begin{proof}The direction i)$\Rightarrow$ ii) and iii) follows from the factorization $\operatorname{id}\ten \Psi(\si)=\Phi\ten \id \Big( \operatorname{id}\ten \Psi(\rho)\Big)$ and \[\norm{\Phi\ten \id:L_1(\M)\hten\M^{op}\to L_1(\M)\hten\M^{op}}{}\le \norm{\Phi: L_1(\M)\to  L_1(\M)}{cb}=1\pl.\]
Let $C(\rho)$ be the convex set from Lemma \ref{convex}
\[ C(\rho)=\{\Phi\ten\operatorname{id}(\rho) | \Phi:L_1(\M)\to L_1(\M), CPTP\}\pl\]
for some bipartite density operator $\rho$.
Suppose by way of contradiction that $\si\notin C$. Because $C(\rho)$ is closed with respect to the weak topology induced by $\M^{op}\ten_{min}\N^{op}$, by the Hahn-Banach theorem
 there exists $x_1\in \M\ten_{min} \N$ such that
\[ \Re \pl \tau(\si x_1)> Re \pl \sup_\Phi \tau( \Phi\ten\operatorname{id} (\rho) x_1 )\pl.\]
We can replace $x_1$ with a finite tensor $x_2=\sum_{j}a_j\ten b_j\in \M\ten \N$ such that $\norm{x_1-x_2}{}<\epsilon$ is small enough and
\[\Re \pl \tau(\si x_2)>\sup_{\Phi }\Re \pl \tau( \operatorname{id}\ten \Phi(\rho)x_2).\]
Take $x_3=(x_2+x_2^*)/2$ be the real part of $x_2$:
\begin{align}\label{5}
 x_3=&\frac{1}{2}(x_2+x_2^*)=\frac{1}{2}\sum_{j}(a_j\ten b_j+a_j^*\ten b_j^*)\nonumber\\
=&\frac{1}{4}\Big(\sum_{j}(a_j+a_j^*)\ten (b_j+b_j^*)+\sum_j i(a_j-a_j^*)\ten (-i)(b_j-b_j^*)\Big),
\end{align}
which is a finite sum of tensor products of self-adjoint elements. Since
$\si$ and $\Phi\ten \id (\rho)$ are positive,
\[ \tau(\si x_3)=\Re \pl \tr(\si x_2)>\sup_{\Phi}\Re \pl \tau( \operatorname{id}\ten \Phi(\rho)x_2)=\sup_{\Phi} \tau( \operatorname{id}\ten \Phi(\rho)x_3).\]
For each $j$,
\[a_j\ten b_j+\norm{a_j}{}\norm{b_j}{}1\ten 1=\frac{1}{2}\Big((a_j+\norm{a_j}{}1)\ten (b_j+\norm{b_j}{}1)+ (\norm{a_j}{}(1-a_j))\ten (\norm{b_j}{}(1-b_j))\Big).\]
is a sum of tensor products of positive elements.
Take $K=\sum_{j}\norm{a_j}{}\norm{b_j}{}$. Then $x_4=x_3+K 1\ten 1\in B(H_A)\ten B(H_B)$ is a sum of tensor products of positive elements. Since $\tau(\operatorname{id}\ten \Phi(\rho))= \tau( \si)=1$, we have
\begin{align}\tau(\si x_4)=\tau(\si x_3)+K>\sup_\Phi \tau( \operatorname{id}\ten \Phi(\rho)x_3)+K\ge \sup_\Phi \tau( \operatorname{id}\ten \Phi(\rho)x_4)\pl. \label{1}\end{align}
The opposite element $x_4^{op}\in \M^{op}\overline{\ten}\N^{op}$ corresponds to a CP map
$T\in CB(L_1(\N),\M^{op})$. Note that  $\id\ten T(\si)\in L_1(\M)$ and  $\id\ten T(\rho)\in L_1(\M)\hten \M^{op}$. We have by Proposition \ref{attain}
\begin{align*}\tau(x_4\si)=\lan T, \si\ran&=\lan \id_\M, \id\ten T(\si)\ran \le \norm{\id\ten T(\si)}{L_1(\M)\hten \M^{op}}, \textnormal{ and}\\
\vspace{5pt}
\sup_{\Phi\pl \text{CPTP}}\tau (x_4 \Phi\ten \id(\rho))&=\sup_{\Phi}\lan T, \Phi\ten \id(\rho)\ran=\sup_{\Phi}\lan \Phi, \id\ten T(\rho)\ran \\ &= \norm{\id\ten T(\rho)}{L_1(\M)\hten \M^{op}}.
\end{align*}
Here the bracket is the pairing for $(L_1(\M)\hten \M^{op})^*\cong CB(\M^{op},\M^{op})$ and $\Phi: L_1(\M)\to L_1(\M)$ is a normal map in $CB(\M^{op},\M^{op})$. Then the inequality \eqref{1} implies that
\[ \norm{\id\ten T(\si)}{L_1(\M)\hten \M^{op}}>\norm{\id\ten T(\rho)}{L_1(\M)\hten \M^{op}} \pl\]
which violates ii). This proves the direction ii)$\Rightarrow$ i). For the direction iii)$\Rightarrow$ i), we shall further modify $T$ to get a CPTP map. There exists a projection $e\in \M$ such that $\tau_\M(e)<\infty$ and $\norm{(e\ten 1)\si (e\ten 1)-\si}{1}< \epsilon$. Then for small enough $\epsilon$ we have
\begin{align}\label{2}\tr(\si (e\ten 1)x_4 (e\ten 1)) >\tr(\si x_4)-\epsilon>\sup_\Phi \tr( \operatorname{id}\ten \Phi(\rho)x_4) \pl.\end{align}
Take ${x_5}:=(e\ten 1)x_4 (e\ten 1)=\sum_{j=1}^nc_j\ten d_j\in e\M e\ten \N$ as a finite sum of tensor product of positive operators. Then $x_5^{op}\in e\M e^{op}\ten \N^{op}$ corresponds to the CP map $T_1: L_1(\N)\to e\M e^{op}$ given by
\[ T_1(\omega)=\sum_{j=1}^n\tau_\N(d_j \omega) c_j\pl.\]
By \eqref{2}, we have
\begin{align*}\tr(\si x_5)&>\sup_{\Phi \pl \textnormal{CPTP}} \tr( \operatorname{id}\ten \Phi(\rho)x_4)=\sup_{\Phi \pl \textnormal{CPTP}} \lan \Phi, \id\ten T(\rho) \ran
=\norm{\id\ten T(\rho)}{L_1(\M)\hten \M^{op}}.
\end{align*}
Take the map $T_1(\cdot)=e T(\cdot) e$. Because the map $y\mapsto e y e$ is a complete contraction from $\M^{op}$ to $e\M e^{op}$, we have
\begin{align*}
\norm{\id\ten T(\rho)}{L_1(\M)\hten \M^{op}}\ge \norm{(1\ten e)\id\ten T(\rho)(1\ten e)}{L_1(\M)\hten \M^{op}}=\norm{\id\ten T_1(\rho)}{L_1(\M)\hten e\M e^{op}}\pl.
\end{align*}
On the other hand,
\begin{align*}\tr(\si x_5)=\lan \id_\M , \id\ten T_1(\si) \ran\le \sup_{\Phi \pl \textnormal{CPTP}} \lan \Phi, \id\ten T_1(\si) \ran=\norm{\id\ten T_1(\si)}{L_1(\M)\hten e\M e^{op}}\pl.
\end{align*}
Thus $T_1:L_1(\M)\to e\M e^{op}$ is a CP and CB map and
\begin{align}\norm{\id\ten T_1(\si)}{L_1(\M)\hten e\M e^{op}}>\norm{\id\ten T_1(\rho)}{L_1(\M)\hten e\M e^{op}}\pl.\label{3}\end{align}
Note that $e\M e^{op}\subset L_1(e\M e^{op})$ because $\tau_\M(e)<\infty$. Since $T_1$ is CP and finite rank, we have
\[ \norm{T_1:L_1(\N)\to L_1(e\M e^{op})}{cb}=\norm{T_1:L_1(\N)\to L_1(e\M e^{op})}{}<\infty\pl.\]
Then $T_2=\norm{T_1:L_1(\N)\to L_1(e\M e^{op})}{}^{-1}T_1$ is CPTNI and satisfies  the inequality \eqref{3}. Finally, we modify $T_2$ to be trace preserving.

Denote by $\rho_\M=\id\ten \tau_\N(\rho)$ and $\rho_\N=\tau_\M\ten \id(\rho)$ the reduced density operator of $\rho$ and similarly for $\si$. For the case $\rho_\N=\si_\N$, we define $T_3=T_2+T_0$ where
$T_2(x)=\big(\tr(x)-\tr(\Psi(x))\big) \frac{e}{\tau_\M(e)}$. Then $T_3:L_1(\M)\to L_1(e\M e^{op})$ is CPTP.  We have
\begin{align*}&\id\ten T_3(\rho)=\id\ten T_2(\rho)+ \frac{\la_1}{\tau_\M(e)} \rho_\M\ten e\pl ,\\
&\id\ten T_3(\si)=\id\ten T_2(\si)+ \frac{\la_2}{\tau_\M(e)} \si_\M\ten e\pl ,
\end{align*}
where $\la_1=\tr(\rho_\N)-\tr(T_2(\rho_\N))$ is equal to $\la_2=\tr(\si_\N)-\tr(T_2(\si_\N))$. Note that for any density operator $\omega\in L_1(\M)$ and $\la>0$
\begin{align*}\id\ten T_3(\rho)=\id\ten T_2(\rho)+ \frac{\la_1}{\tau_\M(e)} \rho_1\ten e \le \la \omega \ten e \pl\Longleftrightarrow  \pl \id\ten T_2(\rho)\le (\la\omega-\frac{\la_1}{\tau_\M(e)}\rho_1)\ten e\pl.
\end{align*}
Therefore we have
\begin{align*}\norm{\id\ten T_3(\rho)}{L_1(\M)\hten e\M e^{op}}=&\norm{\id\ten T_2(\rho)}{L_1(\M)\hten e\M e^{op}}+\frac{\la_1}{\tau_\M(e)} \\  <&
\norm{\id\ten T_2(\si)}{L_1(\M)\hten e\M e^{op}}+\frac{\la_1}{\tau_\M(e)}=\norm{\id\ten T_3(\si)}{L_1(\M)\hten e\M e^{op}}\pl.
\end{align*}
Thus $T_3$ is a CPTP map that violates the condition iii). For the case $\rho_\N\neq \si_\N$, we choose $q_1 \in \N $ to be projection onto the support of $(\si_\N-\rho_\N)_+$ and $q_2=1-q_1$. We define the CPTP map $T_4:L_1(\N)\to \M^{op}$ as
\[T_4(x)=\tr_\N(q x)\frac{e_0}{\tau_\M(e_0)}+\tr_\N(q_2x)\frac{e}{\tau_\M(e)}\]
where $e_0< e$ is a projection (by choosing $e$ large enough, we can always assume such $e_0$ exists). Denote $\si_{\M,j}= \operatorname{id}\ten \tr_\N\big((1\ten q_j)\si\big)$ and $\rho_{\M,j}=\operatorname{id}\ten \tr_\N\big((1\ten q_j)\si\big)$ with $j=1,2$. Note that
\begin{align*}&\tr_\M(\si_{\M,1})+ \tr_\M(\si_{\M,2})=\tr_\M(\si_{\M})=1\pl, \tr_\M(\rho_{\M,1})+ \tr_\M(\rho_{\M,2})=\tr_\M(\rho_{\M})=1\pl, \textnormal{ and}\\ &\tr_\M(\si_{\M,1})-\tr_\M(\rho_{\M,1})=\tau\big((1\ten q_j)(\si-\rho)\big)=\tr_\N\big((\si_\N-\rho_\N)q_1\big)>0.
\end{align*}
Therefore,
\begin{align*}
\norm{\operatorname{id}\ten T_4(\si)}{L_1(\M)\hten e\M e^{op}}&=\norm{\si_{\M,1}\ten \frac{e_0}{\tau_\M(e_0)}+ \si_{\M,2}\ten \frac{e}{\tau_\M(e)}}{L_1(\M)\hten e\M e^{op}}
\\ &=\frac{\tr_\M(\si_{\M,1})}{\tau_\M(e_0)}+\frac{\tr_\M(\si_{\M,2})}{\tau_\M(e)}
\\ &> \frac{\tr_\M(\rho_{\M,1})}{\tau_\M(e_0)}+\frac{\tr_\M(\rho_{\M,2})}{\tau_\M(e)}=\norm{\operatorname{id}\ten T_4(\rho)}{L_1(\M)\hten e\M e^{op}}.
\end{align*}
Note that both $T_3$ and $T_4$ are entanglement-breaking.
Then in both case, we reach a contradiction to condition iii). This proves iii)$\Rightarrow$i).
\end{proof}
\begin{rem}{ \rm  In the above work, the only result that uses the injectivity of $\M$ is Proposition \ref{attain}. In fact, Theorem \ref{von} holds for any von Neumann algebras $\M$ for which the conclusions of Proposition  \ref{attain} hold (in particular, if the identify map satisfies item (iii)), even for Type III cases. Proposition  \ref{attain} uses the equivalence between $L_1(\M,L_\infty(\N))$ and $L_1(\M)\hten\N$ for semifinite injective $\M$. It is possible to extend the definition of $L_1(\M,L_\infty(\N))$ and its connection to $L_1(\M)\hten \N$ for non-tracial $\M$. (See \cite{JungeXu} for the case of $L_1(\M,l_\infty)$ for general $\M$.) Nevertheless, beyond the injective case the projective norm $L_1(\M)\hten\N$ loses its connection to the conditional $H_{min}$ entropy by Theorem \ref{injective}.}\end{rem}

We shall now discuss the special case of $\N=l_\infty$. Let $\{\rho_i\}$ and $\{\si_i\}$ be two families of density operators in $L_1(\M)$. Consider the bipartite density operator $\rho,\si\in L_1(\M)\hten l_1\cong l_1(L_1(\M))$ given by
\[\rho=(\la_i\rho_i)_i\pl, \pl   \si=(\la_i\si_i)_i \pl,\]
where $\la_i> 0, \sum_{i=1}^\infty\la_i=1$ is a probability distribution. Then there exists a CPTP map such that $\si=\Phi\ten \id_{l_1}(\rho)$ if and only if
there exists a CPTP map $\Phi$ such that $\si_i=\Phi(\rho_i)$ for each $i$. The latter statement, called the quantum interpolation problem in \cite{tracial}, concerns the convertibility from one family of density operators to another using a quantum process (CPTP map). For finite families of finite dimensional density operators, it was shown in \cite{tracial} that the quantum interpolation problem is solvable by semi-definite programming (SDP). The $H_{min}$ characterization of quantum interpolation problem was used in \cite{gour18} as a key lemma to prove the bipartite matrix case and has applications in the study of  quantum thermal processes. A similar theorem for finite families of self-adjoint operators is obtained in \cite[Theorem 7.6]{tracial}. We will discuss the connection in Section \ref{tracialset}. The following theorem is an extension in two ways: it addresses infinite sequences and density operators on von Neumann algebras.
\begin{theorem}\label{state}Let $\M$ be an injective semi-finite von Neumann algebra. Let $\{\rho_i\}_{i\in \mathbb{N}}$ and $\{\si_i\}_{i\in \mathbb{N}}$ be two countable families of density operators in $L_1(\M)$. TFAE
\begin{enumerate}
\item[i)]there exists a CPTP map such that $\Phi(\rho_i)=\si_i$ for all $i\in \mathbb{N}$
\item[ii)]for any finitely supported probability distribution $(\la_i)_{i\in \mathbb{N}}$ and any set of density operators $\{\omega_i\}\in L_1(\M^{op})\cap \M^{op}$
\[ \norm{\sum_i\la_i \rho_i\ten \omega_i}{L_1(\M)\hten \M^{op}}\le \norm{\sum_i\la_i \si_i\ten \omega_i}{L_1(\M)\hten \M^{op}}\pl.\]
\end{enumerate}
\end{theorem}
\begin{proof} Choose a probability distribution $(\mu_i)_{i\in \mathbb{N}}$ such that $\mu_i>0$ for each $i\in \mathbb{N}$.  Let $\rho=(\mu_i\rho_i)$ and $\si=(\mu_i\si_i)$ be density operators in $L_1(\M)\hten l_1\cong l_1(L_1(\M))$. Then i)$\Rightarrow$ ii) again follows from the factorization $\Phi\ten\id(\rho)=\si$ and
\[\norm{\Phi:L_1(\M)\to L_1(\M)}{cb}\le 1\pl.\]
Assume that such $\Phi$ does not exists. Then by Theorem \ref{von}
there exists a CPTP map $\Psi:l_1^\infty\to L_1(\M^{op})\cap \M^{op}$ such that
\[ \norm{\operatorname{id}\ten \Psi(\si)}{L_1(\M)\hten \M ^{op}}>\norm{\operatorname{id}\ten \Psi(\rho)}{L_1(\M)\hten \M ^{op}}\pl.\]
Note that the map $\Psi$ constructed is also CB from $l_1^\infty$ to $\M^{op}$. We can choose $N$ such that $\sum_{i>N}\mu_i<\epsilon$. Write $\rho_N=(\rho_i)_{i\le N}\oplus 0$ and $\si_N=(\si_i)_{i\le N}\oplus 0$ as  the corresponding truncated sequences. Then
\[\norm{\operatorname{id}\ten \Psi(\si)-\operatorname{id}\ten \Psi(\si_N)}{L_1(\M)\hten \M ^{op}}\le \norm{\si-\si_N}{L_1(\M)\hten l_1}\le \sum_{i>N}\mu_i<\epsilon\pl.\]
For large enough $N$, we have
\begin{align*}
\norm{\operatorname{id}\ten \Psi(\si_N)}{L_1(\M)\hten \M ^{op}}\ge & \norm{\operatorname{id}\ten \Psi(\si)}{L_1(\M)\hten \M ^{op}}-\epsilon\\ >&\norm{\operatorname{id}\ten \Psi(\rho)}{L_1(\M)\hten \M ^{op}}\\ \ge&
\norm{\operatorname{id}\ten \Psi(\rho_N)}{L_1(\M)\hten \M ^{op}}\pl.
\end{align*}
Write $\omega_i=\Psi(e_i)$ where $e_i$ is the standard basis of $l_1$. We have \[\operatorname{id}\ten \Psi(\si_N)=\sum_{i\le N}\mu_i\si_i\ten \omega_i\pl, \operatorname{id}\ten \Psi(\rho)=\sum_{i\le N}\mu_i\rho_i\ten \omega_i\pl.\]
Renormalizing the coefficient $\la_i=\mu_i(\sum_{i=1}^N\mu_i)^{-1}$, we have a violation of ii). This completes the proof.
\end{proof}

Note that the condition ii) above only concerns finite subsets of $\{\rho_i\}$ and $\{\si_i\}$. This leads to the following ``compactness'' result. It says that to ask whether there is a CPTP map that sends an infinite family of density operators to another infinite family of density operators,
it suffices to check  the convertibility for every finite subfamily of the two infinite families.
\begin{cor}\label{cor:finitesub}
Let $\{\rho_i\}_{i\in \mathbb{N}}$ and $\{\si_i\}_{i\in \mathbb{N}}$ be two infinite families of density operators in $L_1(\M)$. There exists a CPTP map $\Phi$ such that $\Phi(\rho_i)=\si_i$ for all $i\in \mathbb{N}$ if and only if for any finite subset $I\subset \mathbb{N}$, there exists a CPTP map $\Phi_I(\rho_i)=\si_i$ for all $i\in I$.
\end{cor}
\subsection{Channel factorization} \label{channel factorization}

The dual picture of quantum majorization is channel factorization: given two CPTP maps $T,S$, determine if there exists a third  CPTP $\Phi$ such that $\Phi\circ T=S$. Such a factorization relation for two CPTP maps has many implications in quantum information theory. In particular, the channel $T$ has larger capacity than $S$ for various communication task. For a finite dimensional CPTP map $\Phi:M_n\to M_m$, its Choi matrix is
\[\chi_\Phi=\sum_{i,j=1}^n e_{ij}\ten \Phi(e_{ij})\pl \]
where $e_{ij}$ are the matrix unit in $M_n$. As noted in \cite{gour18}, for two CPTP map $S,T:M_n\to M_m$, there exists a CPTP $\Phi$ such that $\Phi\circ T=S$ if and only if
there exists a CPTP $\Phi$ such that $\id\ten \Phi(\chi_T)=\chi_S$. So in finite dimensions channel factorization corresponds to quantum majorization of Choi matrices. However, in the infinite dimensional case, such a correspondence fails because the Choi matrix of a CPTP map is never a density operator (since its trace is unbounded). We shall use again the duality $CB(L_1(\M),L_1(\M))\subset (L_1(\M)\hten\M^{op})^*$ to give a characterization of channel factorization in infinite dimensions and von Neumann algebras. We start with a lemma.

\begin{lemma}\label{convex2}Let $T:L_1(\N)\to L_1(\M)$ be a CPTP map. Define the set of CPTP maps
\begin{align*}&C_{post}(T)=\{\Phi\circ T\pl |\pl  \Phi:L_1(\M)\to L_1(\M) \pl \text{CPTP} \pl \}\pl, \\ &C_{pre}(T)=\{T\circ \Phi\pl |\pl  \Phi:L_1(\M)\to L_1(\M) \pl \text{CPTP} \pl \}\pl.\end{align*}
Then both $C_{post}(T)$ and $C_{pre}(T)$ are weak$^*$-closed in $(L_1(\N)\hten \M^{op})^*=CB(\M^{op},\N^{op})$. Namely, both set are closed in the point-weak topology.
\end{lemma}
\begin{proof}We first argue for $C_{post}(T)$. Let
$(\Phi_\al)$ be a net such that $\Phi_\al\circ T\to S$ in the weak$^*$-topology. That is, for any $x\in L_1(\N), y\in \M$
\begin{align}\label{4}\lim_\al\tau_\M(y\Phi_\al\circ T(x))= \tau_\M(yS(x))\pl.\end{align}
Let $(\Phi_\beta)$ be a sub-net such that $\Phi_\beta\to \Phi$ for some $\Phi:L_1(\N)\to (\M^{op})^*$ in the weak$^*$-topology $CB(L_1(\N), (\M^{op})^*)\cong (L_1(\N)\hten \M^{op})^*$. Note that
\[CB(L_1(\N), (\M^{op})^*)\cong CB(\M^{op},\N^{op})\] by taking adjoint. The map $\Phi^\dag$ is UCP because for any positive $x\in L_1(\N)$
\[\tau_\M(x\Phi^\dag(1))=\lim_{\al}\tau_\M(x\Phi_\al^\dag(1))=\lim_\al\tau_\M(\Phi_\al(x))=\tau_\N(x)\]
We have $\Phi_\beta\circ T\to \Phi\circ T$ because for any $x\in L_1(\N), y\in \M$
\[\lim_\al\tau_\M(y\Phi_\beta\circ T(x))=\lim_\al\tau_\M(\Phi_\beta^\dag(y)T(x))= \tau_\M(\Phi^\dag(y)T(x))=\Phi\circ T(x)(y^{op})\pl.\]
where $\Phi\circ T(x)\in (\M^{op})^*$. Then by \eqref{4}, $\Phi\circ T(x)=S(x)\in L_1(\M)$. This implies $\Phi_n\circ T=S$ for $\Phi_n:L_1(\M)\to L_1(\M)$ being the normal part of $\Phi$.
Since $\Phi_n^\dag$ is normal CP and sub-unital, $\Phi_n$ is CPTNI. Define $\Phi_0(\rho)=\tau_\M(\Phi_n(\rho)-\rho)\si$ where $\si$ is some density operator. Then $\tilde{\Phi}=\Phi_n+\Phi_0$ is CPTP. Moreover, $\Phi_0\circ T=0$ because both $\tilde{\Phi}\circ T$ and $\Phi_n\circ T=S$ are CPTP. Thus, we obtain $\tilde{\Phi}\circ T=\Phi_n\circ T=S$.

For $C_{pre}(T)$, let $\Psi_\al$ be a net such that $T\circ\Phi_\al \to S$ in the weak$^*$-topology. Let $\Psi_\beta$ be a sub-net of $\Psi_\al$ such that $\Psi_\beta\to \Psi$ for some
$\Phi\in CB(L_1(\N), (\M^{op})^*)$. For any $x\in L_1(\N)$ and $y\in \M$,
\[\lim_{\beta}\tau(yT\circ \Psi_\beta(x))=\lim_{\beta}\tau(T^\dag(y) \Psi_\beta(x))=\Psi(x)(T^\dag(y))=T^{\dag\dag}\circ \Psi(x)(y)\]
This means $T\circ \Psi_\beta\to T^{\dag\dag}\circ \Psi$ in the weak$^*$-topo of $CB(L_1(\N), (\M^{op})^*)$. Let $\Psi_n$ be the normal part of $\Psi$. Since $T^{\dag\dag}|_{L_1(\M)}=T$, we have
\[S=T^{\dag\dag}\circ \Psi_n=T^{\dag\dag}|_{L_1(\M)}\circ \Psi_n=T\circ \Psi_n\pl.\]
The argument to modify $\Phi_n$ to be CPTP is similar.
\end{proof}

We say a bipartite density operator $\rho\in L_1(\M\overline{\ten}\N)$ is \emph{separable} if $\rho$ can be written as $\rho=\sum_{j=1}^{\infty}\la_j\omega_j\ten \si_j$,
for some $\la_j\ge 0, \sum_{j=1}^\infty\la_j=1$ and $\omega_j\in L_1(\M),\si_j\in L_1(\N)$ are density operators.
\begin{theorem}\label{pre} Assume that $\M$ is injective. Let $T,S:L_1(\N)\to L_1(\M)$ be two CPTP maps. TFAE
\begin{enumerate}
\item[i)] there exists a CPTP $\Phi:L_1(\M)\to L_1(\M)$ such that $\Phi\circ T= S$
\item[ii)]for any projection $e\in\M$ with $\tau(e)<\infty$ and any separable density operator $\rho\in L_1(\N)\ten e\M e^{op} $,
    \[ \norm{T\ten \operatorname{id}(\rho)}{L_1(\M)\hten e\M e^{op} }\ge \norm{S\ten \operatorname{id}(\rho)}{L_1(\M)\hten e\M e^{op} }\]
\end{enumerate}
\end{theorem}
\begin{proof}i)$\Rightarrow$ ii) follows from $(\Phi\circ T)\ten \id(\rho)=S\ten \id (\rho)$. For ii) $\Rightarrow$ i), we again argue by contradiction. Suppose $S\notin C_{post}(T)=\{\Phi\circ T| \pl \Phi \pl \text{CPTP} \}$. Then by Lemma \ref{convex2},
there exists
$x_1\in L_1(\N)\hten \M^{op}$ such that
\[Re \lan  S,x_1\ran> Re\sup_{\Phi \pl \textnormal{CPTNI}} \lan \Phi\circ T, x_1\ran\pl.\]
We can replace $x_1$ by a finite tensor sum
$x_2=\sum_{j=1}^n a_j\ten b_j$ with $\norm{x_1-x_2}{L_1(\N)\hten \M^{op}}$ small enough. Moreover, following the same argument in \eqref{5}, $a_j\in L_1(\N)$ and $b_j\in \M^{op}$ can be self-adjoint. Note that for any $\omega\in L_1(\N)$,
\[\lan S, \omega\ten 1\ran=\tr_\M(S(\omega))=\tr_\N(\omega)=\tr_\M(\Phi\circ T(\omega))=\lan \Phi\circ T,  \omega \ten 1\ran\]
because $S$ and $\Phi\circ T$ are trace preserving. Then we can replace $x_2$ by
\[x_3=\sum_{j}a_j\ten b_j+\norm{b_j}{} (|a_j|\ten 1)=\sum_{j} (a_j)_+\ten (\norm{b_j}{}1+b_j) +(a_j)_-\ten (\norm{b_j}{}1-b_j)\pl  \]
which is a finite tensor of positive elements. Let $e\in \M$ be a projection with finite trace such that
\[\left|\sum_{j}\tau_\M(b_j^{op}S(a_j))- \sum_{j}\tau_\M(e b_j^{op} eS(a_j))\right|<\epsilon\pl.\]
Take $x_4=(1\ten e)x_3 (1\ten e)$. We have for small $\epsilon$
\begin{align} \lan  S,x_4\ran> \lan  S,x_3\ran-\epsilon > \sup_{\Phi } \lan \Phi\circ T, x_3\ran. \label{ineq}\end{align}
Now we reinterpret the duality pairing
\begin{align*}
\lan  S,x_4\ran=\lan  \operatorname{id}, S\ten \id(x_4)\ran \le& \sup_{\Phi\pl CPTP}\lan \Phi, S\ten \id(x_4)\ran= \norm{S\ten \id(x_4)}{L_1(\M)\hten e\M e^{op}}\pl.
\\
\sup_{\Phi \pl CPTP} \lan \Phi\circ T, x_3\ran= &\sup_{\Phi\pl CPTP}\lan \Phi,  T\ten \id(x_3)\ran= \norm{ T\ten id (x_3)}{L_1(\M)\hten e\M e^{op}}\\ \ge& \norm{ T\ten id (x_4)}{L_1(\M)\hten e\M e^{op}}
\end{align*}
Thus we have a violation of ii),
\[\norm{ S\ten id (x_4)}{L_1(\M)\hten e\M e^{op}}>\norm{ T\ten id (x_4)}{L_1(\M)\hten e\M e^{op}}\pl.\]
Here $x_4\in L_1(\M)\hten e\M e^{op}$ is a finite tensor of positive element with finite trace. Replacing $x_4$ by its normalization, we get a separable density operator. This completes the proof.
\end{proof}
The above theorem gives the characterization for ``post''-factorization. Similarly,
we consider the ``pre''-factorization, which is equivalent to the ``post''-factorization of normal UCP maps.

\begin{theorem}\label{post} Assume that $\M$ is injective. Let $T,S:L_1(\M)\to L_1(\N)$ be two CPTP maps. TFAE
\begin{enumerate}
\item[i)]there exists a CPTP   $\Phi:L_1(\M)\to L_1(\M)$ such that $T\circ \Phi= S$,
  \item[ii)]for any positive $x\in \N^{op}\ten \M $,
    \[ \norm{T^\dag\ten \operatorname{id}(x)}{\M^{op}\overline{\ten} \M  }\le \norm{S^\dag\ten \operatorname{id}(x)}{\M^{op}\overline{\ten} \M}\]
\end{enumerate}
\end{theorem}
\begin{proof}By taking the adjoint, $\Phi^{\dag}\circ T^{\dag}=S^{\dag}$ as normal UCP maps. Then i)$\Rightarrow$ ii) follows from
\[\norm{S^{\dag}\ten \operatorname{id} (x)}{\infty}=\norm{\Phi^{\dag}\circ T^{\dag}\ten \operatorname{id}(x)}{\infty}\le \norm{T^{\dag}\ten \operatorname{id}(x)}{\infty}\pl.\]
For ii) $\Rightarrow$ i), suppose $S\notin C_{pre}(T):=\{T\circ \Phi\,| \pl \Phi \pl \text{CPTP} \}$. By the same argument as for Theorem \ref{pre}, there exists a finite tensor
$x_2=\sum_{j}a_j\ten b_j\in L_1(\M)\hten \N^{op}$ with $a_j,b_j$ positive such that
\[\lan  S,x_2\ran> \sup_{\Phi \pl CPTP} \lan T\circ \Phi, x_2\ran\pl.\]
Then we choose a finite trace projection $e\in \M$ such that $ea_je\in \M$ are bounded and for
$x_3=(e\ten 1)x_2(e\ten 1)=\sum_{j}ea_je\ten b_j$,
\begin{align} \lan  S,x_3\ran> \lan  S,x_2\ran-\epsilon> \sup_{\Phi \pl CPTP} \lan T\circ \Phi, x_2\ran \label{sep2}\end{align}
Reinterpret the pairings
\begin{align*}
&\lan  S,x_3\ran=\lan \operatorname{id}, \id\ten S^\dag (x_3)\ran \le \sup_{\Phi \pl CPTP} \lan \Phi, \id\ten S^\dag (x_3)\ran = \norm{\id\ten S^\dag (x_3)}{L_1(\M)\hten \M^{op}}\pl\textnormal{ and}
\\ & \lan  T\circ \Phi,x_2\ran= \sup_{\Phi \pl CPTP} \lan \Phi, \id\ten T^\dag (x_2)\ran = \norm{\id\ten T^\dag (x_2)}{L_1(\M)\hten \M^{op}}\ge \norm{\id\ten T^\dag (x_3)}{L_1(\M)\hten \M^{op}}
\end{align*}
This implies
\begin{align*}
\norm{\id\ten S^\dag(x_3)}{L_1(\M)\hten \M^{op} }>\norm{\id\ten T^\dag(x_2)}{L_1(\M)\hten \M^{op} }\ge\norm{\id\ten T^\dag(x_3)}{L_1(\M)\hten \M^{op} }.
\end{align*}
Because $\operatorname{id}\ten T^\dag (x_3)$ is a positive operator in $e\M e\ten \M^{op}$, we have
\begin{align*}
\norm{\operatorname{id}\ten T^\dag (x_3)}{L_1(e\M e)\hten \M^{op} }=\inf_{\si} \norm{( \si^{-\frac{1}{2}}\ten 1)\operatorname{id}\ten T^\dag (x_3)( \si^{-\frac{1}{2}}\ten 1)}{\infty}\pl
\end{align*}
where the infimum is over all invertible density operators $\si\in e\M e$. It suffices to consider invertible $\si$ with $\norm{\si^{-1}}{e\M e}<\infty$ because we can always replace $\si$ by an invertible density operator $\tilde{\si}=(\si+\delta e)$. Thus we choose an invertible density operator $\si\in e\M e$ such that
\begin{align*}
\norm{( \si^{-\frac{1}{2}}\ten 1)\operatorname{id} \ten T^\dag(x_3) ( \si^{-\frac{1}{2}}\ten 1)}{e\M e \overline{\ten} \M} <&\norm{\operatorname{id}\ten T^\dag (x_3)}{L_1(e\M e)\hten \M^{op} }+\epsilon \\ <&\norm{\operatorname{id}\ten S^\dag (x_3)}{L_1(e\M e)\hten \M^{op} }\\ \le & \norm{( \si^{-\frac{1}{2}}\ten 1)\operatorname{id} \ten S^\dag(x_3) ( \si^{-\frac{1}{2}}\ten 1)}{e\M e \overline{\ten} \M}.
\end{align*}
Then $x_4=( \si^{-\frac{1}{2}}\ten 1)x_3 ( \si^{-\frac{1}{2}}\ten 1)$ is positive in $\M \ten \N^{op}$, and we have
\begin{align*}
\norm{\operatorname{id}\ten T^\dag (x_4) }{\M \overline{\ten} \M^{op}} < \norm{\operatorname{id}\ten S^\dag (x_4)}{\M \overline{\ten} \M^{op}}\pl.
\end{align*}
which is a violation to condition ii). This proves ii)$\Rightarrow $ i).
\end{proof}


\subsection{Approximate case} \label{approximate case}
In \cite{jencovachannel}, Jen\v{c}ov\'a gives a characterization for the approximate post-channel factorization in finite dimensions that
\[\inf_{\Phi \textnormal{ CPTP}}\norm{S-\Phi\circ T}{cb}<\delta\]
is small but nonzero. Inspired by Jen\v{c}ov\'a's work, we consider the approximate case of quantum majorization. The following lemma is an analogue of \cite[Proposition 1]{jencovachannel}.
\begin{lemma}\label{12}
i) Let $\rho,\si$ be two density operators in $L_1(\M)$. Then
\[\frac{1}{2}\norm{\rho-\si}{1}=\sup \{\tau(x(\rho-\si))\pl|\pl x\ge 0, \norm{x}{\infty}\le 1 \}\pl.\]
ii) Let $\M$ be injective. Let $T,S:L_1(\M)\to L_1(\M)$ be two CPTP maps. Then
\[\frac{1}{2}\norm{T-S}{cb}=\sup \{\lan T-S,\rho\ran \pl|\pl \rho\ge 0, \norm{\rho}{L_1(\M,L_\infty(\M^{op}))}\le 1 \}\pl.\]
\end{lemma}
\begin{proof}
For i), note that
\[x=x^*,\norm{x}{}\le 1\Longleftrightarrow x+1\ge 0, \norm{x+1}{}\le 2.\]
Since $\tau(\rho-\si)=0$,
\begin{align*}
\norm{\rho-\si}{1}=&\sup\{ \Re\pl \tau(x(\rho-\si))| \norm{x}{\infty}\le 1 \}
\\ =&\sup\{  \tau(x(\rho-\si))| \norm{x}{\infty}\le 1, x\pl  \text{self-adjoint} \}
\\ =&\sup\{  \tau((x+1)(\rho-\si))| \norm{x}{\infty}\le 1, x\pl  \text{self-adjoint} \}
\\ =&\sup\{  \tau(y(\rho-\si))| \norm{y}{\infty}\le 2, y\ge 0 \}
\\=& 2 \sup\{  \tau(y(\rho-\si))| \norm{y}{\infty}\le 1, y\ge 0 \}.
\end{align*}
For ii), let $x$ be self-adjoint and satisfy $\norm{x}{L_1(\M,L_\infty(\M^{op}))}< 1$. There exists a density operator $\si\in L_1(\M)$ such that $-\si \ten 1\le x\le \si \ten 1$. Then
\[x+\si\ten 1\ge 0, \norm{x+\si\ten 1}{L_1(\M,L_\infty(\M^{op}))}\le 2 \pl. \]
Conversely, let $y\ge 0, \norm{y}{L_1(\M,L_\infty(\M^{op}))}< 2$. There there exists a density operator $\si\in L_1(\M)$ such that
$0\le y\le 2\si \ten 1$. Then
\[ -\si\ten 1\le y-\si\ten 1\le \si\ten 1\pl, \norm{y-\si\ten 1}{L_1(\M,L_\infty(\M^{op}))}\le 1.\]
Since $\M$ is injective, we have $L_1(\M,L_\infty(\M^{op}))\cong L_1(\M)\hten \M^{op}$. Then using the fact that $\lan T-S, \si \ten 1\ran =\tau(T(\si))-\tau(S(\si))=0$, we have
\begin{align*}
\norm{T-S}{cb}=&\sup\{ \Re\pl \lan T-S, x \ran | \norm{x}{L_1(\M,L_\infty(\M^{op}))}< 1  \}
\\ =&\sup\{  \lan T-S, x \ran | \norm{x}{L_1(\M,L_\infty(\M^{op}))}< 1, x=x^*  \}
\\ =&\sup\{  \lan T-S, x \ran | \norm{x}{L_1(\M,L_\infty(\M^{op}))}< 2, x\ge 0  \}
\\ =&2\sup\{  \lan T-S, x\ran | \norm{x}{L_1(\M,L_\infty(\M^{op}))}< 1, x\ge 0  \}.\qedhere
\end{align*}
\end{proof}
\begin{theorem}\label{apro1}Let $\M,\N$ be semi-finite von Neumanna algebras and $\M$ be injective and $\tau_\M(1)=+\infty$. Suppose $\rho$ and $\si$ are two density operators in $L_1(\M\overline{\ten} \N)$ such that $\tau_\M\ten \id(\rho)=\tau_\M\ten \id(\si)$. TFAE
\begin{enumerate}
\item[i)] $\displaystyle \inf_{\Phi\pl  \textnormal{CPTP}} \norm{\si-\Phi\ten \id(\rho)}{1}\le \delta $.
\item[ii)] for any CPTP map $\Psi:L_1(\N)\to L_1(\M^{op})\cap \M ^{op}$, we have
\[\norm{\id\ten \Psi (\si)}{L_1(\M)\hten \M ^{op}}\le \norm{\id\ten \Psi (\rho)}{L_1(\M)\hten \M ^{op}}+\frac{\delta}{2}\norm{\Psi:L_1(\N)\to \M^{op}}{cb}\]
\end{enumerate}
\end{theorem}
\begin{proof}For a CPTP $\Psi$, we can choose $R:L_1(\M)\to L_1(\M)$ CPTNI such that
such that
\[ \lan R, \id\ten \Psi (\si)\ran\ge \norm{\id\ten \Psi (\si)}{L_1(\M)\hten \M^{op}}-\epsilon\pl.\]
Then
\begin{align*}
&\norm{\id\ten \Psi (\si)}{L_1(\M)\hten \M^{op}} \le \epsilon+\lan R, \id\ten \Psi (\si)\ran\\
\le &\epsilon+\lan R, \Phi\ten \Psi (\rho)\ran+ \lan R, \id\ten \Psi (\si)-\Phi\ten \Psi (\rho)\ran
\\
\le &\epsilon+\lan R\circ \Phi, \id\ten \Psi(\rho)\ran+ \lan \Psi^\dag\circ R,  \si-\Phi\ten \id (\rho)\ran
\\
\le &\epsilon+\norm{\id\ten \id(\rho)}{L_1(\M)\hten \M^{op}}+\frac{1}{2}\norm{\Psi^\dag\circ R}{cb}\norm{\si-\Phi\ten \id (\rho)}{1}
\\
\le &\epsilon+\norm{\id\ten \id(\rho)}{L_1(\M)\hten \M^{op}}+\frac{1}{2}\norm{\Psi^\dag}{cb}\norm{\si-\Phi\ten \id (\rho)}{1}
\end{align*}
where in the second last inequality we used Lemma \ref{12} i).
Then i)$\Rightarrow$ ii) follows from taking the infimum over all CPTP $\Phi$ and $\epsilon \to 0$.
Conversely, suppose $\displaystyle\inf_{\Phi \pl  \textnormal{CPTP}} \norm{\si-\Phi\ten \id(\rho)}{1}>\delta$.  For $x\in \N\overline{\ten} \M$,
\[\tau\big(x(\si-\Phi\ten \id(\rho))\big)=\lan T, \si-\id\ten \Phi(\rho)\ran\pl, \]
where $T$ is the map corresponding to $x^{op}$ via the Effros-Ruan isomorphism \[CB(L_1(\N),\M^{op})\cong \N^{op}\overline{\ten} \M^{op}\pl.\] Because this pairing is linear for both $T$ and $\Phi$, we have by min-max theorem,
\begin{align*}
\delta<&\inf_{\Phi  \textnormal{ CPTP}} \norm{\si-\Phi\ten \id(\rho)}{1}
\\ =& 2\inf_{\Phi  \textnormal{ CPTP}} \sup_{T \pl CP, \norm{T}{cb}\le 1}\lan T, \si-\Phi\ten \id(\rho)\ran
\\ =&  2\sup_{T \pl CP, \norm{T}{cb}\le 1} \inf_{\Phi  \textnormal{ CPTP}}\lan T, \si-\Phi\ten \id(\rho)\ran
\\ =&  2\sup_{T \pl CP, \norm{T}{cb}\le 1} \lan T, \si\ran- \sup_{\Phi  \textnormal{ CPTP}}
\lan T,\Phi\ten \id(\rho)\ran
\end{align*}
Rescaling the above inequality, there exist CP and CB $T:L_1(\N)\to \M^{op}$ such that
\[\lan T, \si\ran- \sup_{\Phi  \textnormal{ CPTP}}
\lan T,  \Phi\ten \id(\rho)\ran>\frac{\delta}{2}\norm{T:L_1(\N)\to \M^{op}}{cb}\pl.\]
For a projection $e\in \M$, denote the map $T_e(\cdot)=eT(\cdot)e$.
There exists $e$ with $\tau_\N(e)<\infty$ such that $|\lan T, (e\ten 1)\si(e\ten 1)-\si\ran|$ is small and
\begin{align*}
\lan T_e, \si\ran&=\lan T, (e\ten 1)\si(e\ten 1)\ran\\ &>\sup_{\Phi  \textnormal{ CPTP}}
\lan T, \Phi\ten \id(\rho)\ran+\frac{\delta}{2}\norm{T:L_1(\N)\to \M^{op}}{cb}
\\ &=
\norm{\id\ten T(\rho)}{L_1(\M)\hten \M ^{op}}+\frac{\delta}{2}\norm{T:L_1(\N)\to \M^{op}}{cb}
\\ &\ge
\norm{\id\ten T_e(\rho)}{L_1(\M)\hten e\M e ^{op}}+\frac{\delta}{2}\norm{T_e:L_1(\N)\to e\M e^{op}}{cb}.
\end{align*}
Here we use the fact that
\[\sup_{\Phi  \textnormal{ CPTP}}
\lan T, \Phi\ten \id(\rho)\ran=\sup_{\Phi  \textnormal{ CPTP}}
\lan \Phi, \id\ten T(\rho)\ran=\norm{\id\ten T(\rho)}{L_1(\M)\hten \M ^{op}}\pl \]
and the projection from $\M$ to $e\M e$ is a complete contraction.
Also, we have \[{\lan T_e, \si\ran= \lan \id, \id\ten T_e(\si)\ran\le \norm{\id\ten T_e(\si)}{L_1(\M)\hten e\M e ^{op}}}\pl.\] Therefore, we have a violation of ii) for $T_e: L_1(\N)\to e\M e^{op}$ is CP and CB,
\begin{align}\norm{\id\ten T_e(\si)}{L_1(\M)\hten e\M e ^{op}}>\norm{\id\ten T_e(\rho)}{L_1(\M)\hten e\M e ^{op}}+\frac{\delta}{2}\norm{T_e:L_1(\N)\to e\M e^{op}}{cb}\label{6}\end{align}
By linearity, we can assume $T_e$ is CPTNI. Denote $\rho_\N=\tau_\M\ten \id(\rho)$ and $\si_\N=\tau_\M\ten \id(\si)$. Because $\rho_\N=\si_\N$, we follow the argument in Theorem \ref{von} to replace $T_e$ by
\[\tilde{T}=T_e+T_0\pl, \pl T_0(x)=\frac{\tau_\M(x-T_e(x))}{\tau_\M(e)}e\pl.\]
Note that $\norm{T_0:L_1(\N)\to e\M e^{op}}{cb}=\frac{1}{\tau_\M(e)}$. Then we can always choose $\tau_\M(e)$ large enough such that $\norm{\tilde{T}}{cb}- \norm{T_e}{cb}$ is small and \eqref{6} is satisfied for $\tilde{T}$.
\end{proof}
\begin{rem}{\rm If, in addition, $\inf \{\tau_\M(e_0) |\pl  e_0 \pl \text{nonzero projection}\}=0$, we do not need the assumption
$\rho_\N=\si_\N$ in Theorem~\ref{apro1}. In the case of $\rho_\N\neq\si_\N$, by the corresponding discussion in Theorem \ref{von}, we have a CPTP map $T_1$ such that
\begin{align*}\norm{\id\ten T_1(\si)}{L_1(\M)\hten e\M e ^{op}}-\norm{\id\ten T_1(\rho)}{L_1(\M)\hten e\M e^{op}}>(\frac{1}{\tau_\M(e_0)}-\frac{1}{\tau_\M(e)})\tau_\N\big((\rho_\N-\si_\N)_-\big)\end{align*}
where $e_0\le e$ is a sub-projection. This difference can be arbitrarily large if $\displaystyle \inf_{e_0\neq 0} \tau_\M(e_0)=0$.}
\end{rem}

This following is a generalization of \cite[Theorem 1]{jencovachannel}.
\begin{theorem}\label{apro2}Let $\M,\N$ be semi-finite von Neumanna algebras and $\M$ be injective. Let $S,T:L_1(\N)\to L_1(\M)$ be two CPTP maps. TFAE
\begin{enumerate}
\item[i)] $\displaystyle \inf_{\Phi\pl   \textnormal{CPTP}} \norm{S- \Phi\circ T}{cb}\le \delta$
\item[ii)] for any density operator $\rho\in L_1(\N\overline{\ten}\M ^{op})$, we have
\[\norm{S\ten id (\rho)}{L_1(\M)\hten \M^{op}}\le \norm{ T\ten \id(\rho)}{L_1(\M)\hten \M^{op}}+\frac{\delta}{2}\norm{\rho}{L_1(\M)\hten \M^{op}}\pl.\]
\end{enumerate}
\end{theorem}
\begin{proof}Let $\rho\in L_1(\N\overline{\ten}\M ^{op})$ be a density operator. For any $\epsilon>0$, we can choose $R:L_1(\M)\to L_1(\M)$ CPTNI such that
\[ \lan R, S\ten \id (\rho)\ran\ge \norm{S\ten \id (\rho)}{L_1(\M)\hten \M^{op}}-\epsilon\pl.\]
Then
\begin{align*}
&\norm{S\ten \id (\rho)}{L_1(\M)\hten \M^{op}} \le \epsilon+\lan R, S\ten id (\rho)\ran\\
\le &\epsilon+\lan R, \Phi\circ T\ten \id (\rho)\ran+ \lan R, (S-T)\ten id (\rho)\ran
\\
\le &\epsilon+\norm{\Phi\circ T\ten \id (\rho)}{L_1(\M)\hten \M^{op}}+ \lan S-\Phi\circ T, \id\ten R^\dag (\rho)\ran
\\
\le &\epsilon+\norm{\Phi\circ T\ten \id (\rho)}{L_1(\M)\hten \M^{op}}+ \frac{1}{2}\norm{S-\Phi\circ T}{cb}\norm{\id\ten R^\dag(\rho)}{L_1(\M)\hten \M^{op}}
\\
\le &\epsilon+\norm{\Phi\circ T\ten \id (\rho)}{L_1(\M)\hten \M^{op}}+ \frac{1}{2}\norm{S-\Phi\circ T}{cb}\norm{\rho}{L_1(\M)\hten \M^{op}}
\end{align*}
where in the second last inequality we used Lemma \ref{12} ii).
Then i)$\Rightarrow$ ii) follows from taking the infimum over all CPTP $\Phi$ and $\epsilon \to 0$. For ii) $\Rightarrow$ i), suppose $\inf_{\Phi  \textnormal{CPTP}} \norm{S- \Phi\circ T}{cb}>\delta$. Let us use the shorthand notation
$\norm{\cdot}{1,\infty}=\norm{\cdot}{L_1(\M)\hten \M^{op}}$.
Using the min-max theorem,
\begin{align*}
\delta<&\inf_{\Phi\pl  \textnormal{CPTP}} \norm{S- \Phi\circ T}{cb}
\\ =& 2\inf_{\Phi \pl  \textnormal{CPTP}} \sup_{\rho\ge 0, \norm{\rho}{1,\infty}\le 1}\lan S- \Phi\circ T, \rho\ran
\\ =&  2\sup_{\rho\ge 0, \norm{\rho}{1,\infty}\le 1} \inf_{\Phi  \textnormal{CPTP}}\lan S- \Phi\circ T, \rho\ran
\\ =&  2\sup_{\rho\ge 0, \norm{\rho}{1,\infty}\le 1} \lan \id, S\ten \id(\rho)\ran- \sup_{\Phi  \textnormal{CPTP}}
\lan \Phi, T\ten \id(\rho)\ran
\\ \le & 2 \sup_{\rho\ge 0, \norm{\rho}{1,\infty}\le 1} \sup_{\Phi  \textnormal{CPTP}}\lan \Phi, S\ten \id(\rho)\ran- \sup_{\Phi  \textnormal{CPTP}}
\lan \Phi, T\ten \id(\rho)\ran
\\ =& 2\sup_{\rho\ge 0, \norm{\rho}{1,\infty}\le 1}\norm{S\ten \id(\rho)}{L_1(\M)\hten \M^{op}}-\norm{T\ten \id(\rho)}{L_1(\M)\hten \M^{op}}.
\end{align*}
Thus there exists a positive $\rho\in L_1(\M)\hten \M^{op}$ violating the inequality in ii). One can then replace $\rho$ by a bipartite density operator $\tilde{\rho}$ in $L_1(\M\overline{\ten} \M^{op})$ as in Theorem \ref{pre}.
\end{proof}
\begin{rem} {\rm In Theorem \ref{apro1} $\&$ \ref{apro2}, we cannot reduce condition ii) to entanglement-breaking CPTP maps and respectively separable density operator as in the case for $\delta=0$. This is because Lemma \ref{12} fails when we restrict the pairing to entanglement-breaking or separable elements.}
\end{rem}

\subsection{Results in $B(H)$ setting}\label{BH case}
The results of the previous subsections subsume the case of $B(H)$ where $H$ is infinite dimensional. However, since this is the case most relevant to quantum information theory, we briefly restate some of our results for $B(H)$ in terms of the conditional min entropy $H_{min}$.  $H_{min}(A|B)$ is the sandwiched R\'enyi $p$-version of $H(A|B)$ at $p=\infty$ and the smooth version of $H_{min}(A|B)$ connects to $H(A|B)$ by quantum asymptotic equipartition
property \cite{tomamichel}.
While the operational meaning of $H(A|B)$ is in i.i.d. asymptotic regime, $H_{min}(A|B)$ has many applications in the one shot setting (\cite{tomamichelbook} and reference therein). The following theorem summarizes the results on quantum majorization, state convertibility and channel factorization.

\begin{theorem} Let $H_A,H_B$ be two infinite dimensional Hilbert spaces. The following statements hold.
\begin{enumerate}
\item[i)] For two bipartite density operators $\rho^{AB},\si^{AB}\in S_1(H_A\ten_2 H_B)$, there exists a quantum channel $\Phi:S_1(H_B)\to S_1(H_{B})$ such that $\operatorname{id}_A\ten \Phi(\rho)=\si$ if and only if for any entanglement-breaking channel $\Psi:S_1(H_A)\to S_1(H_{A})$
\[  H_{min}(A|B)_{\Psi\ten \operatorname{id}(\rho)}\le H_{min}(A|B)_{\Psi\ten \operatorname{id}(\si)}\pl.\]
\item[ii)] For two families of density operators $\{\rho_i\}_{i\in \mathbb{N}}$ and $\{\si_i\}_{i\in \mathbb{N}}$ in $B(H_B)$,
    there exists a quantum channel such that $\Phi(\rho_i)=\si_i$ for all $i\in \mathbb{N}$ if and only if for any finitely supported probability distribution $\la_i$ on $\mathbb{N}$ and any set of density operators $\{\omega_i\}\in B(H_A)$
\[ H_{min}(A|B)_{(\sum_i\la_i\omega_i\ten \rho_i)}\le H_{min}(A|B)_{(\sum_i\la_i\omega_i\ten \si_i)}\pl.\]
\item[iii)] For two quantum channels $T,S:S_1(H_B)\to S_1(H_B)$, there exists a quantum channel $\Phi$ such that $\Phi\circ T =S$ if and only if for any separable density operator $\rho\in S_1(H_A\ten_2 H_B)$,
\ \[H_{min}(A|B)_{\operatorname{id}\ten T(\rho)}\le  H_{min}(A|B)_{\operatorname{id}\ten S(\rho)}.\]
\end{enumerate}
\end{theorem}
The above theorem make senses even when $H_{min}$ equals ``$-\infty$''.
We know by Theorem \ref{state} and \ref{pre} that it suffices to consider all finite dimensional $H_A$ in the equivalence ii) and iii). Similarly, for the equivalence it suffices to consider channels $\Psi:S_1(H_A)\to S_1(H_{A'})$ into a finite dimensional $H_A'$. In these situation, $H_{min}$ will always take finite values. In general, $H_{min}(A|B)$ can be ``$-\infty$'', where the inequalities in above theorem are trivially satisfied.

\section{Tracial convex sets in Vector-valued noncommutative $L_1$-space}
\label{tracialset}
In this section, we discuss the analogue of quantum majorization in vector-valued noncommutative $L_1$-space and the connection to the tracial Hahn-Banach Theorem. Let $(\M,\tau)$ be a semifinite von Neumann algebra equipped with a normal faithful semifinite trace $\tau$. Let $E$ be a operator space.
The $E$-valued noncommutative $L_1$-space was introduced by Pisier in \cite{pisier93}. For $x \in \M_0\ten E$ in the algebraic tensor, we define the $L_1(\M,E)$ norm as follows,
\begin{align}\norm{x}{L_1(\M,E)}=\inf \{\norm{a}{L_2(\M)}\norm{b}{L_2(\M)}\norm{y}{\M \ten_{min}E} \pl | \pl x=a\cdot y \cdot b \}, \label{norm1}\end{align}
where the infimum runs over all factorizations $x=a\cdot y \cdot b:=(a\ten 1_E) y(b\ten 1_E)$ with $a,b\in \M_0$ and $y\in \M\ten E$. The space $L_1(\M,E)$ is defined as the norm completion of $\M_0\ten E$. The $L_1(\M,L_\infty(\N))$ space we discussed in the previous section is the special case of $E$ being a von Neumann algebra $\N$. Recall that a von Neumann algebra $\M$ is hyperfinite if
$\M=\overline{\cup \M_\al}$ is the $w^*$-closure of the union of an increasing net of finite dimensional von Nuemann algebras $\M_\al$. It was proved in \cite[Theorem 3.4]{pisier93} that for hyperfinite $\M$,
\begin{align}L_1(\M,E)\cong L_1(\M)\hten E \label{h}\end{align}
isometrically. Namely, for hyperfinite $\M$, the vector-valued noncommutative $L_1$ space is identified with projective tensor product. Following that, we introduce the following definition of a tracial set in $L_1(\M)\hten E$.
\begin{defi}\label{tracial}
A subset $V\subset L_1(\M)\hten E$ is called a contractively tracial set if for any CPTNI map $\Phi:L_1(\M)\to L_1(\M)$ , $\Phi\ten \operatorname{id}_E(V)\subset V$.
\end{defi}
The matrix level tracial sets are discussed in \cite[Section 6.2]{tracial} as the dual concept of matrix convex set. We refer to their definition as matrix tracial set.
\begin{defi}\label{matrixtracial} A matrix contractively tracial set $(V_n)_n$ is a sequence of subsets $V_n\subset M_n(E)$ such that for any CPTNI map $\Phi:M_n\to M_m$, $\Phi\ten \operatorname{id} (V_n)\subset V_m$.
\end{defi}
This definition was considered in \cite{tracial} for finite dimensional $E$. Indeed, for $\dim E=m$, each element in $V_n \subset M_n(E)\cong M_n^m$ can be identified with a finite sequence $(x_j)\in (M_n)^m$. We discuss the relations of these two definitions in the following proposition.
\begin{prop}Let $H$ be an separable Hilbert space and $(e_n)_n$ be a sequence of projections such that $\dim (e_nH)= n$ and $e_n\to 1$ weakly. Identify $M_n\cong S_1(eH_n)$ as subspace of $S_1(H)$.
\begin{enumerate}
\item[i)] Given a contractively tracial set $ V\subset S_1(H)\hten E$, then the set
\[V[n]=e_n\cdot V \cdot e_n\]
forms a matrix contractively tracial set such that $\overline{\cup_n V[n]}=\overline{V^{\norm{\pl\cdot\pl}{}}}$.
\item[ii)]Given a matrix contractively tracial set $(V_n)\subset M_n(E)$, then the set
\[V=\overline{(\cup_n V_n)^{\norm{\pl\cdot\pl}{}}}\subset S_1(H)\hten E\]
is a closed contractively tracial set such that $V[n]=\overline{V_n}$ .
\end{enumerate}
\end{prop}
\begin{proof}i) Let $e\in B(H)$ be a projection. Because the map $\rho \to e\rho e$ is CPTNI on $S_1(H)$, $x\in V$ implies that $e\cdot x \cdot e \in V$. Then for any $\Phi:M_n\to M_m$ CPTNI, $\Phi \ten \operatorname{id} (e_n\cdot  x\cdot  e_n)\in V[m] \subset V$. Thus $(V[n])_n$ is a matrix contractively tracial set. Moreover, for any $\epsilon>0$ and $x\in S_1(H)\hten E$, $\lim_{n}\norm{e_n\cdot x \cdot e_n -x}{S_1(H)\hten E}\to 0$. Then $\overline{V^{\norm{\pl\cdot\pl}{}}}\subset \overline{\cup_n V[n]^{\norm{\pl\cdot\pl}{}}}$ and the other inclusion follows from $V[n]\subset V$.

ii) Let $x\in V_n$. For $\Phi:S_1(H)\to S_1(H)$ CPTNI and $\rho\in V_n$, we find that \[e_m\cdot \Phi\ten \operatorname{id}(\rho)\cdot e_m \in V_m\]
because $x\to e_m\Phi(x) e_m $ can be viewed as a CPTNI map from $M_n$ to $M_m$. Let $x_k\in V_{n(k)}$ be a sequence such that $x_k\to x$ in $S_1(H)\hten E$. Then $\Phi\ten \operatorname{id}(x_m)\to \Phi\ten \operatorname{id} (x)$, which implies $\Phi\ten \operatorname{id} (x)\in V$. This verifies that $V$ is contractively tracial. In particular, the fact that $e_n\cdot x_k\cdot e_n$ converges to $e_n\cdot x\cdot e_n$ implies that $V[n]\subset\overline{V_n} $.
\end{proof}
The above proposition shows that Definition \ref{tracial} and Definition \ref{matrixtracial} are closely related for the case $\M=B(H)$. In particular, they coincide for closed sets. It is easy to see that the convex hull of a contractively tracial set is again contractively tracial. In general, contractively tracial sets are not necessary convex.

The next theorem is the tracial Hahn-Banach separation theorem for convex contractively tracial sets. For matrix contractively tracial sets with $\dim E<\infty$, this was obtained in \cite[Theorem 7.6]{tracial}. Using projective tensor product, we can now consider semi-finite injective $\M$ and a general operator space $E$.
\begin{theorem}\label{tracialHB}Let $\M$ be an injective semifinite von Neumann algebra. Let $V$ be a closed convex contractively tracial set in $L_1(\M)\hten E$ and $x\in L_1(\M)\hten E$. Then $x\notin V$ if and only if there exists a CB map $T:E\to \M^{op}$ such that for each $y\in V$, these exists a density operator $\omega_y\in L_1(\M)$ depending on $y$ such that \[\Re \pl \operatorname{id}\ten T(y)\le \omega_y\ten 1\]
and for any density operator $\omega$,
\[ \Re \pl \operatorname{id}\ten T(x)\nleq\omega\ten 1 \]
\end{theorem}
\begin{proof}The ``if'' direction is trivial. For the other direction,
suppose $\si\notin V$. Using the duality $L_1(\M)\hten E^*=CB(E,\M^{op})$, there exists a CB map $T:E\to \M^{op}$
\[\Re \pl \lan T ,x\ran >\sup_{\rho\in V} \Re \pl \lan T, y\ran. \]
Reinterpreting the dual pairing and using the Proposition \ref{attain},
\begin{align*}\Re \pl \lan T,x\ran&=\Re \pl \lan \operatorname{id}_\M, \operatorname{id}\ten T(x)\ran\le \sup_{\Phi \pl  \textnormal{CPTNI}} \Re \pl \lan \Phi, \operatorname{id}\ten T(x)\ran \\ &= \inf \{\tau(\omega)|  \Re\pl  \operatorname{id}\ten T(x)\le \omega \ten 1,\pl, \omega \ge 0\}\pl.\end{align*}
On the other hand, because $V$ is contractively tracial,
\begin{align*}
\sup_{\rho\in V} \Re \pl \lan T,y\ran& \ge \sup_{\rho\in V, \Phi  \textnormal{ CPTNI}} \Re \pl \lan T, \Phi\ten \operatorname{id}( y)\ran
\\&=\sup_{\rho\in V} \inf \{\tau(x)| \Re\pl  T\ten \operatorname{id}(y)\le x \pl, x \ge 0\}
\end{align*}
Take $\la$ such that $\Re \pl \lan T ,x \ran >\la >\sup_{y\in V} \Re \pl \lan T, y\ran$. Then for the map $\tilde{T}=\frac{1}{\la}T$,
\[ \sup_{y\in V} \inf \{\tau(\omega)| \Re\pl  \tilde T\ten \operatorname{id}(y)\le \omega\ten 1  \pl, \omega \ge 0\}<1< \inf \{\tau(\omega)| \Re\pl  \tilde T\ten \operatorname{id}( x)\le \omega\ten 1 \pl, \omega \ge 0\}\]
which completes the proof.
\end{proof}
Using the similar idea, we obtain a variant of Effros-Winkler's separation theorem \cite{EW97}. Recall a CP map $\Phi$ is sub-unital if $\Phi(1)\le 1$.
\begin{theorem}\label{weak}
Let $E$ be a operator space.
Let $V\subset M_n(E)$ be a closed convex set such $\Phi\ten \operatorname{id}(V)\subset V$ for any CP sub-unital $\Phi:M_n\to M_n$. Then $x\notin V$ if and only if there exists a map $T:E\to M_n$ such that for each $y\in V$, there exists a density operator $\omega_y\in M_n$ depending on $y$ such that
\[ \Re \pl \operatorname{id}\ten T(y)\le 1\ten \omega_y \pl,\]
 and for any density operator $\omega$,
\[\Re \pl \operatorname{id}\ten T(x)\nleq 1\ten \omega\pl.\]
\end{theorem}
\begin{proof}Suppose $x\notin V$. Because $M_n$ is finite dimensional,
we have $M_n(E)^*=S_1^n\hten E^*$. Then there exists an element $T\in E^*\hten S_1^n$ such that
\begin{align}\label{sep3}   \Re \pl \lan T, x\ran > \sup_{y \in V} \Re\pl \lan T, y\ran\pl.\end{align}
We identify $T\in E^*\hten S_1^n$  with a map $T:E \to S_1^n$. Then the pairing on the left hand side of (\ref{sep3}) can be rewritten as
\[\Re \pl \lan T, x\ran= \Re \lan \operatorname{id}_{M_n}, \operatorname{id} \ten T(x)\ran \le \inf \{\tau(\omega)| \Re\pl  \operatorname{id}\ten T(x)\le 1\ten \omega  \pl, \omega \ge 0\} \pl.\]
Here the second pairing is between $CB(M_n,M_n)=(M_n\hten S_1^n)^*$. For the right hand side of (\ref{sep3}),
\begin{eqnarray*}
\sup_{y\in V}\Re \pl \lan T, y\ran&=&\sup_{y\in V}\sup_{\Phi\pl \text{CP sub-unital}}\Re \pl \lan T, \Phi\ten \operatorname{id}(y)\ran
=\sup_{y\in V}\sup_{\Phi}\Re \lan \Phi, \operatorname{id} \ten T(x)\ran  \\
&\le& \sup_{y\in V}\inf \{\tau(\omega)| \Re\pl   \operatorname{id}\ten T(y)\le 1\ten \omega  \pl, \omega \ge 0\} \pl.
\end{eqnarray*}
Then the assertion follows from the inequality \eqref{sep3}.
\end{proof}
Recall that a contractively matrix convex set is a sequence $(V_n)\subset M_n(E)$ such that i) for any CP sub-unital $\Phi:M_m\to M_n$, $\Phi\ten \operatorname{id}(V_m)\subset V_n$; and ii) for any $a\in V_m,b\in V_n$, $a\oplus b\in V_{n+m}$. Effros-Winkler's theorem stated for matrix convex set admits a stronger separation: there exists a density operator $\omega$ uniform for all $y$ such that $\Re \pl \operatorname{id}\ten T(y)\le 1\ten \omega$. A similar lemma for tracial sets was given in \cite[Lemma 7.4]{tracial}.
The above Theorem \ref{weak} leads to a weaker separation because we consider convex sets closed under CP sub-unital maps without assumption ii).

\section{Norm separations on projective tensor product}\label{sec:opspace}
In this section, we discuss the analogue of quantum majorization on projective tensor product. Recall that a operator space $G$ is \emph{1-locally reflexive} if for any finite dimensional operator space $G$, we have the complete isometry
\[CB(E,G^{**})\cong CB(E,G)^{**}\pl.\]
It is clear from the definition that $G=G^{**}$ is reflexive implies that $G$ is $1$-locally reflexive. It was proved by Effros, Junge, and Ruan \cite{EJR} that the predual of von Neumann algebras are $1$-locally reflexive. Another property needed in our discussion is completely contractive approximation property (CCAP). A operator space $E$ is  \emph{CCAP} if there exists a net of finite rank completely contractive maps $\Phi_\al:E\to E$ such that for any $x$, $\Phi_\al(x)\to x$ in norm. In the setting of operator spaces, this is an analog of $w^*$-CPAP (or injectivity).

The following lemma shows that these two properties combined give the desired norm attaining property similar to Proposition \ref{attain}.
Throughout this section, we write $CB$ for completely bounded and $CC$ for completely contractive.
\begin{lemma}\label{wdense}Let $E$ be CCAP. Then $CB(E,G)\subset CB(E,G^{**})$ is $w^*$-dense in the sense of $CB(E,G^{**})=(E\hten G^*)^*$. If in additional $G$ is $1$-locally reflexive, then
\[\norm{\rho}{F\hten G^*}=\sup\{ \Re \pl \lan \Psi, \rho\ran | \Psi:E\to G \pl \text{ CC}\}.\]
\end{lemma}
\begin{proof} Let $\Phi_\al: E\to E$ be a net of CC maps such that $\Phi_{\al}(x)\to x$ in norm for any $x\in E$. For $\rho \in E\hten G^*$ with $\norm{\rho}{E\hten G^*}=1$, we can choose a finite tensor sum $\rho_0=\sum_{j=1}^n x_j\ten y_j$ such that
$\norm{\rho-\rho_0}{E\hten G^*}\le \epsilon$. Then for $T: E\to G^{**}$ with $\norm{T}{cb}=1$, there exists an $\al$ such that
\begin{align*} |\lan T\circ \Phi_\al-T, \rho \ran| &\le |\lan T\circ \Phi_\al- T, \rho-\rho_0 \ran|+ |\lan T\circ \Phi_\al-T, \rho_0 \ran|\\
&\le |\lan T\circ \Phi_\al- T, \rho-\rho_0 \ran|+ |\lan T, \Phi\ten \operatorname{id}(\rho_0)-\rho_0 \ran|\le 2\epsilon+\epsilon \pl.\end{align*}
Let $E_\al$ be the range of $\Phi_\al$ as a finite dimensional subspace of $E$ and $T|_{E_\al}\in CB(E_\al,G^{**})$ be the restriction of $T$ to $E_\al$. There exists $T_\al\in CB(E_\al,G)$ such that \[|\lan T_\al-T, \Phi_\al\ten \operatorname{id} (\rho_0)\ran |=|\lan (T_\al-T)\circ \Phi_\al, \rho_0\ran|\le \epsilon \pl.\]
Therefore $T_\al\circ \Phi_\al:E\to G$ is CB and
\begin{align*}|\lan T_\al\circ \Phi_\al-T, \rho\ran|
 \le & |\lan T\circ \Phi_\al-T, \rho\ran|  +|\lan (T_\al-T)\circ \Phi_\al, \rho-\rho_0\ran|+|\lan (T_\al-T)\circ \Phi_\al, \rho_0\ran|
\\ \le & 3\epsilon+2\epsilon + \epsilon = 6\epsilon\pl  \end{align*}
which proves the $w^*$-density of $CB(E,G)\subset CB(E,G^{**})$. If $G$ is $1$-locally reflexive, $T_\al$ and $T_\al\circ \Phi_\al$ can be CC because of the isometry $CB(E_\al,G^{**})\cong CB(E_\al,G)^{**}$\end{proof}

The following theorem is the analog of quantum majorization and channel factorization in the abstract operator space setting.
\begin{theorem}\label{normsep}Let $E,F,G$ be operator spaces. Suppose one of the following condition holds: \begin{enumerate}
\item[a)]$G$ is reflexive;
\item[b)] $G$ is $1$-locally reflective and $F$ is CCAP
\end{enumerate}
Then the following two statements hold:
\begin{enumerate}
\item[i)] For $\rho\in E\hten F$ and $\si\in E\hten G$,  there exists a sequence of CC maps $u_n:F\to G$ such that $\operatorname{id}\ten u_n(\rho)\to \si$ in the norm of $E\hat{\ten} G$ if and only if for any CB map $v:E\to G^* $, \[\norm{v\ten \operatorname{id}(\rho)}{G^*\hat{\ten}F}\ge \norm{v\ten \operatorname{id}(\si)}{G^*\hat{\ten}G}.\]
\item[ii)]For $T\in CB(E,F)$ and $S\in CB(E,G)$, there exists a net of CC $u_\al:F\to G$ such that $u_\al\circ T\to S$ in the point-weak topology if and only if for any $x\in E\ten G^* $, \[\norm{T\ten \operatorname{id}(x)}{F\hten G^*}\ge \norm{S\ten \operatorname{id}(x)}{G\hten G^*}\pl.\]
\end{enumerate}
\end{theorem}
\begin{proof}i) The ``only if'' direction is easy. For the if part, consider the norm-closed convex set
\[C(\rho)=\overline{\{\operatorname{id}\ten u(\rho)| u :F\to G, CC \}}\subset E\hten G.\]
If $\si \notin C$, there exists a $v\in CB(E,G^*)=(E\hten G)^*$ such that
\[\Re \pl \lan v ,\si \ran >\sup_u \Re\pl \lan v, \operatorname{id}\ten u(\rho)\ran \pl.\]
Let $\iota_G: G\to G^{**}$ be the embedding. Note that
\begin{align*}  &\Re\pl \lan v, \si \ran=\Re\pl \lan \iota_G ,v\ten \operatorname{id} (\si) \ran\le \norm{v\ten \operatorname{id} (\si) }{G^*\ten G}\pl,\\
& \sup_u \Re\pl \lan v, \operatorname{id}\ten u(\rho)\ran = \sup_{u} \Re\pl \lan u, v\ten \operatorname{id}(\rho)\ran =\norm{v\ten \operatorname{id}(\rho)}{G^*\hten F}
\end{align*}
where the last equality follows Lemma \ref{wdense}.

ii) Suppose $u_\al$ is a net of CC maps  such that $u_\al\circ T\to S$ in the point-weak topology. Then for any $R\in CB(G^*, G^*)=(G\hten G^*)^*$ and $x\in E\hten G^*$
\[\lim_\al \lan R, u_\al\circ T\ten \operatorname{id}(x)\ran=\lim_\al \lan u_\al\circ T, \operatorname{id}\ten R(x)\ran=\lan S, \operatorname{id}\ten R(x)\ran=\lan R, S\ten \operatorname{id}(x)\ran\pl \]
which implies $\norm{T\ten \operatorname{id}(x)}{F\hten G^*}\ge \norm{S\ten \operatorname{id}(x)}{G\hten G^*}$. For the converse,
consider the $w^*$-closure of convex set
\[C(T)=\overline{\{u\circ T| u :F\to G, CC \}^{w}}\subset CB(E,G^{**})=(E\hten G^{*})^*.\]
If $S \notin C$, there exists a $\rho \in E\hten G^*$ such that
\[\Re \pl \lan S, \rho \ran > \sup_u \Re \pl \lan u\circ T, \rho\ran\pl.\]
By a density argument, we can further assume $\rho\in E\ten G^*$ in the algebraic tensor product. Note that
\begin{align*}  & \Re\pl \lan S, \rho \ran=\Re\pl \lan \iota ,S\ten \operatorname{id} (\si) \ran\le \norm{S\ten \operatorname{id} (\si) }{G\ten G^*}\pl,\\
& \sup_u \Re\pl \lan u\circ T, \rho\ran =\sup_u \Re\pl \lan u, T\ten \operatorname{id}(\rho)\ran= \norm{\rho}{F\hten G^*}
\end{align*}
where again the last equality uses Lemma \ref{wdense}.
\end{proof}

The following proposition discusses the case when the limits in above theorem can be replaced by equality.
\begin{prop}Let $E,F,G$ be operator spaces. Let $T\in CB(E,F)$ and $\rho\in E\hten F$. Suppose $G=(G_*)^*$ is a dual space. Then
  $\{u\circ T\pl | \pl  u:F\to G,\pl CC \}$ is $w^*$-closed in $CB(E,G)$.
If, in addition, $E$ is CCAP or $G$ is reflexive,   $\{\operatorname{id}\ten u(\rho)\pl |\pl  u:F\to G,\pl CC \}$ is norm-closed in $E\hten G$.
\end{prop}
\begin{proof} To prove the first statement, let $u_k:F\to G$ be a sequence of CC maps such that $\lim_k u_k\circ T=S$ in the $w^*$-topology of $CB(E,G)=(E\hten G_*)^*$. Because $CB(F,G)=(F\hten G_*)^*$, we choose $u$ as $w^*$-limit of $(u_k)$ such that the subsequence $u_{k_i}\to u$. Then $u_{k_i}\circ T\to u\circ T$ in the point $w^*$-topology hence $S=T$. For the second statement, we assume $E$ is CCAP or $G$ is reflexive. Let $u_k:F\to G$ be a sequence of CC such that $ \operatorname{id}\ten u_k(\rho)\to \si$ in the norm of $E\hat{\ten} G$. Choose a subsequence $u_{k_i}\to u$ in the $w^*$-topology for some CC $u$. For for any $T\in CB(E,G_*)$,
\[\lim_{i}\lan T, \operatorname{id}\ten u_{k_i}(\rho)\ran= \lim_{i}\lan u_{k_i}, T\ten \operatorname{id}(\rho)\ran=\lan u, T\ten \operatorname{id}(\rho)\ran= \lan T, u\ten \operatorname{id}(\rho) \ran.\]
Thus $\operatorname{id}\ten u_{k_i}(\rho)\to \operatorname{id}\ten u(\rho)$ in $E\hten G$ with the topology induced by $CB(E,G_*)\subset CB(E,G^*)$. Note that by Lemma \ref{wdense}, this topology is separating. Hence we have $\si=\lim_{i} \operatorname{id}\ten u_{k_i}(\rho)= \operatorname{id}\ten u(\rho)$.
\end{proof}

Theorem \ref{normsep} also holds for Banach space tensor products. We can replace the operator space  concepts with their Banach space counterparts: replace
``operator spaces'' by ``Banach spaces'', ``CB (resp. CC)'' by ``bounded (resp. contractive)'' and ``CCAP'' by ``contractive approximation property (or $1$-AP)''. Moreover, all Banach spaces have $1$-local reflexivity. We refer to the book \cite{lindenstrauss96} for definitions of the above mentioned Banach space concepts. Here we state the result analogous to Theorem \ref{normsep}. Let $\ten_\pi$ denote the Banach space projective tensor product and $B(E,F)$ be the set of bounded maps from Banach space $E$ to $F$.
\begin{theorem}Let $E,F,G$ be Banach spaces. Suppose one of the following conditions holds: \begin{enumerate}
\item[a)]$G$ is reflexive;
\item[b)] $F$ is nuclear.
\end{enumerate}
Then the following two statements hold:
\begin{enumerate}
\item[i)]for $\rho\in E\ten_{\pi}F$ and $\si\in E\ten_{\pi} G$, there exists a sequence of contraction map $u_n:F\to G$ such that $\operatorname{id}\ten u_n(\rho)\to \si$ in the norm of $E\ten_\pi G$ if and only if for any bounded map $v:E\to G^* $,
    \[\norm{v\ten \operatorname{id}(\rho)}{G^*\ten_{\pi}F}\ge \norm{v\ten \operatorname{id}(\si)}{G^*\ten_{\pi}G}\pl. \]
\item[ii)]for $T\in B(E,F)$ and $S\in B(E,G)$, there exists a net of contraction $u_\al:F\to G$ such that $u_\al\circ T\to S$ in the point-weak topology if and only if for any $x\in E\ten_\pi G^* $, \[\norm{T\ten \operatorname{id}(x)}{F\ten_\pi G^*}\ge \norm{S\ten \operatorname{id}(x)}{G\ten_\pi G^*}\pl.\]
    \end{enumerate}
\end{theorem}
The proof is identical to Theorem \ref{normsep} and the details are left to the reader.
\vspace{20pt}

{\bf Acknowledgements---}
The authors are grateful to Marius Junge for many helpful discussions, and to Eric Ricard for remarks on the assumption on Theorem \ref{von}. The authors would also like to thank Sasmita Patnaik for useful discussions at the beginning of this project. The authors acknowledge that the approximate case in Section \ref{approximate case} was inspired by a talk given by Anna Jen{\v{c}}ov{\'a} at the Algebraic and Statistical ways into Quantum Resource Theories workshop at BIRS in July 2019.
Satish K.\ Pandey is supported in part at the Technion by a fellowship of the Israel Council for Higher Education. Sarah Plosker is supported by NSERC Discovery Grant number 1174582, the Canada Foundation for Innovation (CFI) grant number 35711, and the Canada Research Chairs (CRC) Program grant number 231250.

\bibliographystyle{amsalpha}
\bibliography{qmajorization}

\begin{bibdiv}
\begin{biblist}

\bib{bds}{article}{
      author={Buscemi, Francesco},
      author={Datta, Nilanjana},
      author={Strelchuk, Sergii},
       title={Game-theoretic characterization of antidegradable channels},
        date={2014},
     journal={Journal of Mathematical Physics},
      volume={55},
      number={9},
       pages={092202},
}

\bib{PB}{article}{
      author={Blecher, David~P},
      author={Paulsen, Vern~I},
       title={Tensor products of operator spaces},
        date={1991},
     journal={Journal of Functional Analysis},
      volume={99},
      number={2},
       pages={262\ndash 292},
}

\bib{buscemi}{article}{
      author={Buscemi, Francesco},
       title={Comparison of quantum statistical models: equivalent conditions
  for sufficiency},
        date={2012},
     journal={Communications in Mathematical Physics},
      volume={310},
      number={3},
       pages={625\ndash 647},
}

\bib{Chefles}{article}{
      author={Chefles, Anthony},
       title={The quantum blackwell theorem and minimum error state
  discrimination},
        date={2009},
     journal={arXiv preprint arXiv:0907.0866},
}

\bib{connes76}{article}{
      author={Connes, Alain},
       title={Classification of injective factors cases {II}$_1$,
  {II}$_\infty$, {III}$_\lambda$, $\lambda \ne$ 1},
        date={1976},
     journal={Annals of Mathematics},
       pages={73\ndash 115},
}

\bib{ER00}{book}{
      author={Effros, Edward~G},
      author={Jin, Ruan~Zhong},
       title={Operator spaces},
   publisher={Clarendon Press},
        date={2000},
}

\bib{EJR}{article}{
      author={Effros, Edward~G},
      author={Junge, Marius},
      author={Ruan, Zhong-Jin},
       title={Integral mappings and the principle of local reflexivity for
  noncommutative {L}$^1$-spaces},
        date={2000},
     journal={Annals of Mathematics},
      volume={151},
      number={1},
       pages={59\ndash 92},
}

\bib{EW97}{article}{
      author={Effros, Edward~G},
      author={Winkler, Soren},
       title={Matrix convexity: operator analogues of the bipolar and
  {H}ahn--{B}anach theorems},
        date={1997},
     journal={Journal of functional analysis},
      volume={144},
      number={1},
       pages={117\ndash 152},
}

\bib{gour18}{article}{
      author={Gour, Gilad},
      author={Jennings, David},
      author={Buscemi, Francesco},
      author={Duan, Runyao},
      author={Marvian, Iman},
       title={Quantum majorization and a complete set of entropic conditions
  for quantum thermodynamics},
        date={2018},
     journal={Nature communications},
      volume={9},
      number={1},
       pages={5352},
}

\bib{gour19}{article}{
      author={Gour, Gilad},
       title={Comparison of quantum channels by superchannels},
        date={2019},
     journal={IEEE Transactions on Information Theory},
}

\bib{haagerup84}{incollection}{
      author={Haagerup, Uffe},
       title={Injectivity and decomposition of completely bounded maps},
        date={1985},
   booktitle={Operator algebras and their connections with topology and ergodic
  theory},
   publisher={Springer},
       pages={170\ndash 222},
}

\bib{hardy}{article}{
      author={Hardy, Godfrey~H},
       title={Some simple inequalities satisfied by convex functions},
        date={1929},
     journal={Messenger Math.},
      volume={58},
       pages={145\ndash 152},
}

\bib{tracial}{article}{
      author={Helton, J~William},
      author={Klep, Igor},
      author={McCullough, Scott},
       title={The tracial {H}ahn-{B}anach theorem, polar duals, matrix convex
  sets, and projections of free spectrahedra},
        date={2014},
     journal={arXiv preprint arXiv:1407.8198},
}

\bib{winter}{article}{
      author={Horodecki, Micha{\l}},
      author={Oppenheim, Jonathan},
      author={Winter, Andreas},
       title={Partial quantum information},
        date={2005},
     journal={Nature},
      volume={436},
      number={7051},
       pages={673},
}

\bib{jencovachannel}{inproceedings}{
      author={Jen{\v{c}}ov{\'a}, Anna},
       title={Comparison of quantum channels and statistical experiments},
organization={IEEE},
        date={2016},
   booktitle={2016 {IEEE} international symposium on information theory
  ({ISIT})},
       pages={2249\ndash 2253},
}

\bib{JungeXu}{article}{
      author={Junge, Marius},
      author={Xu, Quanhua},
       title={Noncommutative maximal ergodic theorems},
        date={2007},
     journal={Journal of the American Mathematical Society},
      volume={20},
      number={2},
       pages={385\ndash 439},
}

\bib{lindenstrauss96}{misc}{
      author={Lindenstrauss, Joram},
      author={Tzafriri, Lior},
       title={Classical {B}anach spaces {I} and {II}, classics in mathematics},
   publisher={Springer-Verlag},
        date={1996},
}

\bib{muller}{article}{
      author={M{\"u}ller-Lennert, Martin},
      author={Dupuis, Fr{\'e}d{\'e}ric},
      author={Szehr, Oleg},
      author={Fehr, Serge},
      author={Tomamichel, Marco},
       title={On quantum r{\'e}nyi entropies: A new generalization and some
  properties},
        date={2013},
     journal={Journal of Mathematical Physics},
      volume={54},
      number={12},
       pages={122203},
}

\bib{marshall}{book}{
      author={Marshall, Albert~W},
      author={Olkin, Ingram},
      author={Arnold, Barry~C},
       title={Inequalities: theory of majorization and its applications},
   publisher={Springer},
        date={1979},
      volume={143},
}

\bib{pisieros}{book}{
      author={Pisier, Gilles},
       title={Introduction to operator space theory},
   publisher={Cambridge University Press},
        date={2003},
      volume={294},
}

\bib{pisier93}{article}{
      author={Pisier, Gilles},
       title={Non-commutative vector valued ${L}_p$-spaces and completely
  $p$-summing maps},
        date={1998},
     journal={Ast{\'e}risque},
}

\bib{Shmaya}{article}{
      author={Shmaya, Eran},
       title={Comparison of information structures and completely positive
  maps},
        date={2005},
     journal={Journal of Physics A: Mathematical and General},
      volume={38},
      number={44},
       pages={9717},
}

\bib{takesaki}{book}{
      author={Takesaki, Masamichi},
       title={Theory of operator algebras. {I}},
   publisher={Springer-Verlag},
        date={1979},
}

\bib{tomamichel}{article}{
      author={Tomamichel, Marco},
      author={Colbeck, Roger},
      author={Renner, Renato},
       title={A fully quantum asymptotic equipartition property},
        date={2009},
     journal={IEEE Transactions on information theory},
      volume={55},
      number={12},
       pages={5840\ndash 5847},
}

\bib{tomamichelbook}{book}{
      author={Tomamichel, Marco},
       title={Quantum information processing with finite resources:
  Mathematical foundations},
   publisher={Springer},
        date={2015},
      volume={5},
}

\end{biblist}
\end{bibdiv}


\end{document}